\documentclass[secthm,seceqn,amsthm,ussrhead,10pt]{amsart}
\usepackage[utf8]{inputenc}
\usepackage[english]{babel}
\usepackage{amssymb,amsmath,amsthm,amsfonts,xcolor,enumerate,hyperref,comment,longtable,cleveref}

\usepackage{times}
\usepackage{cite}
\usepackage{pdflscape}
\usepackage{ulem}
\usepackage[mathcal]{euscript}
\usepackage{tikz}
\usepackage{hyperref}
\usepackage{cancel}
\usepackage{stmaryrd}

\usetikzlibrary{arrows}

\newtheorem{theorem}{Theorem}
\newtheorem{lemma}[theorem]{Lemma}

\usepackage{stmaryrd}
\usepackage{xcolor}

\newcommand{\der}[1]{\operatorname{\mathrm{Der}}{#1}}

\newcommand{\orb}{\mathrm{Orb}}
\newcommand{\D}[2]{\mathbf{D}^{#1}_{#2}}
\newcommand{\A}{\mathbf{A}}
\newcommand{\B}{\mathbf{B}}
\newcommand{\C}{\mathbf{C}}

\newcommand{\cd}[2]{\mathfrak{CD}^{#1}_{#2}}

\newcommand{\lb}{\lambda}

\newcommand{\af}{\alpha}
\newcommand{\bt}{\beta}
\newcommand{\gm}{\gamma}

\begin{document}

{\Large\noindent 
The geometric  classification of nilpotent $\mathfrak{CD}$-algebras
\footnote{ The authors thank  Prof. Dr. Pasha Zusmanovich for   discussions about  $\mathfrak{CD}$-algebras.}
 
}

\ 

   {\bf Ivan Kaygorodov$^{a}$ \&   Mykola Khrypchenko$^{b}$ \\

    \medskip
}

{\tiny

$^{a}$ CMCC, Universidade Federal do ABC, Santo Andr\'e, Brazil

$^{b}$ Departamento de Matem\'atica, Universidade Federal de Santa Catarina, Florian\'opolis, Brazil

\

\smallskip

   E-mail addresses:

    \smallskip

    Ivan Kaygorodov (kaygorodov.ivan@gmail.com)
    
    Mykola Khrypchenko (nskhripchenko@gmail.com)

}

\

\noindent{\bf Abstract}: 
{\it We give a geometric  classification of complex
$4$-dimensional nilpotent $\mathfrak{CD}$-algebras. 
 The corresponding geometric variety has dimension 18 and decomposes into $2$ irreducible components determined by the Zariski closures of a two-parameter family of algebras and a four-parameter family of algebras (see Theorem~\ref{main-geo}). 
 In particular, there are no rigid $4$-dimensional complex nilpotent $\mathfrak{CD}$-algebras.}

\


\ 


\section*{Introduction}

There are many results related to the algebraic and geometric 
classification
of low-dimensional algebras in the varieties of Jordan, Lie, Leibniz and 
Zinbiel algebras;
for algebraic classifications  see, for example, 
\cite{ack, demir,  lisa,   ikm19,  kkk18,cfk20, kpv19, kv16}; 
for geometric classifications and descriptions of degenerations see, for example, 
\cite{ack, bb14, BC99, lisa,gkk2, GRH, GRH2, cfk20, g93, jkk19, ikm19, ikv17, ikv18,wolf2, kh15,  kkk18, kkl20, klp19, kpv19, kppv, kppvcor, kpv, kv16, kv17, kv19, S90,wolf1}.
In the present paper, we give algebraic classification of nilpotent $\mathfrak{CD}$-algebras.
This is a new class of non-associative algebras introduced in \cite{ack}.
The idea of the definition of a $\mathfrak{CD}$-algebra  comes from the following property of Jordan and Lie algebras: {\it the commutator of any pair of multiplication operators is a derivation}.
This gives three identities of degree four,    which reduce to only one identity of degree four in the commutative or anticommutative case.
Commutative and anticommutative  $\mathfrak{CD}$-algebras are related to many interesting varieties of algebras.
Thus, anticommutative  $\mathfrak{CD}$-algebras is a generalization of Lie algebras, 
containing the intersection of Malcev and Sagle algebras as a proper subvariety. 
Moreover, the following intersections of varieties coincide:
Malcev and Sagle algebras; 
Malcev and anticommutative  $\mathfrak{CD}$-algebras; 
Sagle and anticommutative  $\mathfrak{CD}$-algebras.
On the other hand, 
the variety of anticommutative  $\mathfrak{CD}$-algebras is a proper subvariety of 
the varieties of binary Lie algebras 
and almost Lie algebras \cite{kz}.
The variety of anticommutative  $\mathfrak{CD}$-algebras  coincides with the intersection of the varieties of binary Lie algebras and almost Lie algebras.
Commutative  $\mathfrak{CD}$-algebras is a generalization of Jordan algebras, 
which is a generalization of associative commutative algebras.
On the other hand, the variety of commutative  $\mathfrak{CD}$-algebras is  the variety of almost-Jordan algebras (sometimes,  called Lie triple algebras)  \cite{peresi,osborn69,Sidorov_1981} 
and the bigger variety of generalized almost-Jordan algebras \cite{arenas,hl,labra}.
The $n$-ary  version of commutative  $\mathfrak{CD}$-algebras was introduced in a recent paper by 
Kaygorodov, Pozhidaev and Saraiva \cite{kps19}. Commutative  $\mathfrak{CD}$-algebras are also related to assosymmetric algebras  \cite{askar18}.

\paragraph{\bf Motivation and contextualization} 
Given algebras ${\bf A}$ and ${\bf B}$ in the same variety, we write ${\bf A}\to {\bf B}$ and say that ${\bf A}$ {\it degenerates} to ${\bf B}$, or that ${\bf A}$ is a {\it deformation} of ${\bf B}$, if ${\bf B}$ is in the Zariski closure of the orbit of ${\bf A}$ (under the base-change action of the general linear group). The study of degenerations of algebras is very rich and closely related to deformation theory, in the sense of Gerstenhaber \cite{ger63}. It offers an insightful geometric perspective on the subject and has been the object of a lot of research.
In particular, there are many results concerning degenerations of algebras of small dimensions in a  variety defined by a set of identities.
One of the main problems of the {\it geometric classification} of a variety of algebras is a description of its irreducible components. In the case of finitely-many orbits (i.e., isomorphism classes), the irreducible components are determined by the rigid algebras --- algebras whose orbit closure is an irreducible component of the variety under consideration. 
The algebraic classification of complex $4$-dimensional nilpotent $\mathfrak{CD}$-algebras was obtained in \cite{kk20}, and in the present paper we continue the study of the variety by giving its geometric classification.

\paragraph{\bf Main result}
The main result of the paper is Theorem~\ref{main-geo}, in which we prove that the variety of complex $4$-dimensional nilpotent $\mathfrak{CD}$-algebras has $2$ irreducible components of dimensions $15$ and $18.$

\section{Degenerations of algebras}
Given an $n$-dimensional vector space ${\bf V}$, the set ${\rm Hom}({\bf V} \otimes {\bf V},{\bf V}) \cong {\bf V}^* \otimes {\bf V}^* \otimes {\bf V}$ 
is a vector space of dimension $n^3$. This space inherits the structure of the affine variety $\mathbb{C}^{n^3}.$ 
Indeed, let us fix a basis $e_1,\dots,e_n$ of ${\bf V}$. Then any $\mu\in {\rm Hom}({\bf V} \otimes {\bf V},{\bf V})$ is determined by $n^3$ structure constants $c_{i,j}^k\in\mathbb{C}$ such that
$\mu(e_i\otimes e_j)=\sum_{k=1}^nc_{i,j}^ke_k$. A subset of ${\rm Hom}({\bf V} \otimes {\bf V},{\bf V})$ is {\it Zariski-closed} if it can be defined by a set of polynomial equations in the variables $c_{i,j}^k$ ($1\le i,j,k\le n$).

The general linear group ${\rm GL}({\bf V})$ acts by conjugation on the variety ${\rm Hom}({\bf V} \otimes {\bf V},{\bf V})$ of all algebra structures on ${\bf V}$:
$$ (g * \mu )(x\otimes y) = g\mu(g^{-1}x\otimes g^{-1}y),$$ 
for $x,y\in {\bf V}$, $\mu\in {\rm Hom}({\bf V} \otimes {\bf V},{\bf V})$ and $g\in {\rm GL}({\bf V})$. Clearly, the ${\rm GL}({\bf V})$-orbits correspond to the isomorphism classes of algebras structures on ${\bf V}$. Let $T$ be a set of polynomial identities which is invariant under isomorphism. Then the subset $\mathbb{L}(T)\subset {\rm Hom}({\bf V} \otimes {\bf V},{\bf V})$ of the algebra structures on ${\bf V}$ which satisfy the identities in $T$ is ${\rm GL}({\bf V})$-invariant and Zariski-closed. It follows that $\mathbb{L}(T)$ decomposes into ${\rm GL}({\bf V})$-orbits. The ${\rm GL}({\bf V})$-orbit of $\mu\in\mathbb{L}(T)$ is denoted by $O(\mu)$ and its Zariski closure by $\overline{O(\mu)}$.

Let ${\bf A}$ and ${\bf B}$ be two $n$-dimensional algebras satisfying the identities from $T$ and $\mu,\lambda \in \mathbb{L}(T)$ represent ${\bf A}$ and ${\bf B}$ respectively.
We say that ${\bf A}$ {\it degenerates} to ${\bf B}$ and write ${\bf A}\to {\bf B}$ if $\lambda\in\overline{O(\mu)}$.
Note that in this case we have $\overline{O(\lambda)}\subset\overline{O(\mu)}$. Hence, the definition of a degeneration does not depend on the choice of $\mu$ and $\lambda$. If ${\bf A}\to {\bf B}$ and ${\bf A}\not\cong {\bf B}$, then ${\bf A}\to {\bf B}$ is called a {\it proper degeneration}. We write ${\bf A}\not\to {\bf B}$ if $\lambda\not\in\overline{O(\mu)}$ and call this a {\it non-degeneration}. Observe that the dimension of the subvariety $\overline{O(\mu)}$ equals $n^2-\dim\der({\bf A})$. Thus if ${\bf A}\to {\bf B}$ is a proper degeneration, then we must have $\dim\der({\bf A})>\dim\der({\bf B})$.

Let ${\bf A}$ be represented by $\mu\in\mathbb{L}(T)$. Then  ${\bf A}$ is  {\it rigid} in $\mathbb{L}(T)$ if $O(\mu)$ is an open subset of $\mathbb{L}(T)$.
Recall that a subset of a variety is called {\it irreducible} if it cannot be represented as a union of two non-trivial closed subsets. A maximal irreducible closed subset of a variety is called an {\it irreducible component}.
It is well known that any affine variety can be represented as a finite union of its irreducible components in a unique way.
The algebra ${\bf A}$ is rigid in $\mathbb{L}(T)$ if and only if $\overline{O(\mu)}$ is an irreducible component of $\mathbb{L}(T)$.

In the present work we use the methods applied to Lie algebras in \cite{GRH,GRH2}.
To prove 
degenerations, we will construct families of matrices parametrized by $t$. Namely, let ${\bf A}$ and ${\bf B}$ be two algebras represented by the structures $\mu$ and $\lambda$ from $\mathbb{L}(T)$, respectively. Let $e_1,\dots, e_n$ be a basis of ${\bf V}$ and $c_{i,j}^k$ ($1\le i,j,k\le n$) be the structure constants of $\lambda$ in this basis. If there exist $a_i^j(t)\in\mathbb{C}$ ($1\le i,j\le n$, $t\in\mathbb{C}^*$) such that the elements $E_i^t=\sum_{j=1}^na_i^j(t)e_j$ ($1\le i\le n$) form a basis of ${\bf V}$ for any $t\in\mathbb{C}^*$, and the structure constants $c_{i,j}^k(t)$ of $\mu$ in the basis $E_1^t,\dots, E_n^t$ satisfy $\lim\limits_{t\to 0}c_{i,j}^k(t)=c_{i,j}^k$, then ${\bf A}\to {\bf B}$. In this case  $E_1^t,\dots, E_n^t$ is called a {\it parametric basis} for ${\bf A}\to {\bf B}$.

To prove a non-degeneration ${\bf A}\not\to {\bf B}$ we will use the following lemma (see \cite{GRH}).

\begin{lemma}\label{main}
Let $\mathcal{B}$ be a Borel subgroup of ${\rm GL}({\bf V})$ and $\mathcal{R}\subset \mathbb{L}(T)$ be a $\mathcal{B}$-stable closed subset.
If ${\bf A} \to {\bf B}$ and ${\bf A}$ can be represented by $\mu\in\mathcal{R}$ then there is $\lambda\in \mathcal{R}$ that represents ${\bf B}$.
\end{lemma}

In particular, it follows from Lemma \ref{main} that ${\bf A}\not\to {\bf B}$, whenever $\dim({\bf A}^2)<\dim({\bf B}^2)$.

When the number of orbits under the action of ${\rm GL}({\bf V})$ on  $\mathbb{L}(T)$ is finite, the graph of primary degenerations gives the whole picture. In particular, the description of rigid algebras and irreducible components can be easily obtained. Since the variety of $4$-dimensional nilpotent $\mathfrak{CD}$-algebras contains infinitely many non-isomorphic algebras, we have to fulfill some additional work. Let ${\bf A}(*):=\{{\bf A}(\alpha)\}_{\alpha\in I}$ be a family of algebras and ${\bf B}$ be another algebra. Suppose that, for $\alpha\in I$, ${\bf A}(\alpha)$ is represented by a structure $\mu(\alpha)\in\mathbb{L}(T)$ and ${\bf B}$ is represented by a structure $\lambda\in\mathbb{L}(T)$. Then by ${\bf A}(*)\to {\bf B}$ we mean $\lambda\in\overline{\cup\{O(\mu(\alpha))\}_{\alpha\in I}}$, and by ${\bf A}(*)\not\to {\bf B}$ we mean $\lambda\not\in\overline{\cup\{O(\mu(\alpha))\}_{\alpha\in I}}$.

Let ${\bf A}(*)$, ${\bf B}$, $\mu(\alpha)$ ($\alpha\in I$) and $\lambda$ be as above. To prove ${\bf A}(*)\to {\bf B}$, it is enough to construct a family of pairs $(f(t), g(t))$ parametrized by $t\in\mathbb{C}^*$, where $f(t)\in I$ and $g(t)=\left(a_i^j(t)\right)_{i,j}\in {\rm GL}({\bf V})$. Namely, let $e_1,\dots, e_n$ be a basis of ${\bf V}$ and $c_{i,j}^k$ ($1\le i,j,k\le n$) be the structure constants of $\lambda$ in this basis. If we construct $a_i^j:\mathbb{C}^*\to \mathbb{C}$ ($1\le i,j\le n$) and $f: \mathbb{C}^* \to I$ such that $E_i^t=\sum_{j=1}^na_i^j(t)e_j$ ($1\le i\le n$) form a basis of ${\bf V}$ for any  $t\in\mathbb{C}^*$, and the structure constants $c_{i,j}^k(t)$ of $\mu\big(f(t)\big)$ in the basis $E_1^t,\dots, E_n^t$ satisfy $\lim\limits_{t\to 0}c_{i,j}^k(t)=c_{i,j}^k$, then ${\bf A}(*)\to {\bf B}$. In this case, $E_1^t,\dots, E_n^t$ and $f(t)$ are called a {\it parametric basis} and a {\it parametric index} for ${\bf A}(*)\to {\bf B}$, respectively. In the construction of degenerations of this sort, we will write $\mu\big(f(t)\big)\to \lambda$, emphasizing that we are proving the assertion $\mu(*)\to\lambda$ using the parametric index $f(t)$.



\section{The algebraic classification of complex $4$-dimensional nilpotent $\mathfrak{CD}$-algebras}

The algebraic classification  of complex $4$-dimensional nilpotent $\mathfrak{CD}$-algebras was obtained in the following three steps.
\begin{enumerate}
    \item Trivial non-anticommutative $\mathfrak{CD}$-algebras were classified in 
    \cite[Theorem 2.1, Theorem 2.3  and Theorem 2.5]{demir}, and the only trivial anticommutative 
     $\mathfrak{CD}$-algebra is $\mathfrak{CD}_{03}^{4*}$ from~\cite{kk20};
    
    \item Non-trivial non-Leibniz terminal extensions of the algebra $\cd {3*}{04}(\lambda)$ were classified in \cite[1.4.5. 1-dimensional central extensions of ${\bf T}^{3*}_{04}$]{kkp19geo}. All of them are $\mathfrak{CD}$-algebras.
    
    \item Non-trivial Leibniz terminal extensions of $\cd {3*}{04}(\lambda)$ were found in \cite{kppvcor}. All of them are $\mathfrak{CD}$-algebras.
    
    \item The rest of algebras were found in \cite{kk20} and are listed below: 
    
\end{enumerate}

{\tiny
\begin{longtable}{lllll llll}
$\cd 4{01}$&$:$& $e_1 e_1 = e_2$ & $e_2 e_2=e_3$ \\
\hline
$\cd 4{02}$&$:$& $e_1 e_1 = e_2$ & $e_2 e_1= e_3$ & $e_2 e_2=e_3$ \\
\hline$\cd 4{03}$&$:$& $e_1 e_1 = e_2$ & $e_2 e_1=e_3$ \\
\hline$\cd 4{04}(\lambda)$&$:$& $ e_1 e_1 = e_2$ & $e_1 e_2=e_3$ & $e_2 e_1=\lambda e_3$ \\ \hline

$\cd 4{05}$ &: & $e_1 e_1 = e_2$ & $e_2 e_1=e_4$ & $e_2 e_2=e_3$ \\
\hline$\cd 4{06}$ &: & $ e_1 e_1 = e_2$ & $e_1 e_2=e_4$ & $e_2 e_1=e_3$  \\

\hline$\cd 4{07}(\lambda)$&: & $e_1 e_1 = e_2$ & $e_1 e_2=e_4$ & $e_2 e_1=\lambda e_4$ & $e_2 e_2=e_3$ \\\hline
$\cd 4{08}(\alpha)$ &$:$&
$e_1 e_1 = e_2$ & $e_1e_3=e_4$ & $e_2 e_1=e_3$ & $e_2e_2=\alpha e_4$ & $e_3e_1=-2e_4$ \\

\hline$\cd 4{09}$ &$:$& 
$e_1 e_1 = e_2$ & $e_1e_2=e_4$ & $e_1e_3=e_4$ \\
&& $e_2 e_1=e_3$ & $e_2e_2=- e_4$ & $e_3e_1=-2e_4$\\
\hline

$\cd 4{10}(\alpha)$ &$:$&
$e_1 e_1 = e_2$ & $e_1 e_2=e_3$ & $e_1e_3=-e_4$ \\
&& $e_2 e_1=  e_3 +e_4$ & $e_2e_2=\alpha e_4$ & $e_3e_1=- e_4$ \\

\hline$\cd 4{11}(\lambda)$ &$:$& 
$e_1 e_1 = e_2$ & $e_1 e_2=e_3$ & $e_1e_3=(\lambda-2)e_4$ \\
$\lambda \neq1$&& $e_2 e_1=  \lb e_3+e_4$  &
$e_2e_2=-(\lambda+1)^2 e_4$ &$e_3e_1= (1-2\lb) e_4$ \\

\hline$\cd 4{12}(\alpha, \lambda)$ &$:$& 
$e_1 e_1 = e_2$ & $e_1 e_2=e_3$ & $e_1e_3=(\lambda-2)e_4$ \\&& $e_2 e_1=\lambda e_3$ & $e_2e_2=\alpha e_4$ & $e_3e_1=(1-2\lambda) e_4.$ \\
\hline

$\cd {4}{13}(\af)$&$:$& 
$e_1 e_1 = e_2$ & $e_1e_2=e_4$ & $e_1e_3=e_4$& $e_2e_1=e_4$ \\
$\af\neq \frac{1}{2}$&&  $e_2e_3=\alpha e_4$& $e_3e_1=e_4$& \multicolumn{2}{l}{$e_3e_2=(\alpha +1)e_4$}\\

\hline
$\cd {4}{14}(\af, \beta)$&$:$& 
$e_1 e_1 = e_2$ & $e_1 e_2 = e_4$&  $e_1 e_3 = \af e_4$& $e_2 e_1 = e_4$\\ && 
$e_2 e_2 = e_4$& $e_3 e_1 = \af e_4$& $e_3 e_2 = e_4$& $e_3 e_3 =\beta  e_4$& \\

\hline
$\cd {4}{15}(\af)$&$:$& 
$e_1 e_1 = e_2$ &  $e_1 e_2 = \af e_4$&   $e_1 e_3 = e_4$&
 $e_2 e_1 =(\af+1)  e_4$ &  $e_3 e_1 = e_4$\\

\hline
$\cd {4}{16}$&$:$& 
$e_1 e_1 = e_2$ & $e_1e_2=e_4$ & $e_2 e_1 = e_4$\\ && $e_2 e_3 = -\frac{1}{2} e_4$& $e_3 e_2 =\frac{1}{2} e_4$& $e_3 e_3 = e_4$&\\

\hline
$\cd {4}{17}(\af)$&$:$& 
$e_1 e_1 = e_2$ &  $e_1 e_2 = e_4$&  $e_2 e_1 = e_4$&  $e_2 e_3 =\af e_4$&  $e_3 e_2 = (\af+1) e_4$&\\

\hline
$\cd {4}{18}(\af)$&$:$& 
$e_1 e_1 = e_2$ &  $e_1 e_2 =\af  e_4$&   \multicolumn{2}{l}{$e_2 e_1 =(\af+1) e_4$}&   $e_3 e_3 = e_4$&\\

\hline
$\cd {4}{19}$&$:$& 
$e_1 e_1 = e_2$ &   $e_1 e_2 = e_4$&   $e_2 e_1 = e_4$&    $e_3 e_1 = e_4$&
 $e_3 e_3 = e_4$&\\
 
\hline
$\cd {4}{20}$&$:$& 
$e_1 e_1 = e_2$ &    $e_1 e_2 = e_4$&   $e_2 e_1 = e_4$&   $e_3 e_1 = e_4$&\\

\hline
$\cd {4}{21}(\af)$&$:$& 
$e_1 e_1 = e_2$ & $e_1 e_3 =\af e_4$&   $e_2 e_1 = e_4$&  $e_2 e_2 = e_4$&\\
&&  $e_2 e_3 = e_4$&  $e_3 e_1 =\af e_4$&  $e_3 e_2 = e_4$&  $e_3 e_3 = e_4$\\

\hline
$\cd {4}{22}$&$:$& 
$e_1 e_1 = e_2$ &   $e_1 e_3 = e_4$&    $e_2 e_3 = -\frac{1}{2} e_4$\\&&
 $e_3 e_1= e_4$&   $e_3 e_2 =\frac{1}{2} e_4$&   $e_3 e_3 = e_4$&\\

\hline
$\cd {4}{23}(\af)$&$:$& 
  $e_1 e_1 = e_2$&   $e_1 e_3 = e_4$&   $e_2 e_3 = \af e_4$&
  $e_3 e_1 = e_4$&   $e_3 e_2 =(\af +1)e_4$&\\

\hline
$\cd {4}{24}(\af)$&$:$& 
  $e_1 e_1 = e_2$&   $e_1 e_3 = e_4$&    $e_2 e_2 = e_4$\\
  && $e_3 e_1 = e_4$&   $e_3 e_2 = e_4$& $e_3 e_3 =\af  e_4$&\\

\hline
$\cd {4}{25} $&$:$& 
  $e_1 e_1 = e_2$&  $e_1 e_3 = e_4$& $e_2 e_1 = e_4$& $e_3 e_1 = e_4$&    $e_2 e_2 = e_4$&\\

\hline
$\cd {4}{26}(\af)$&$:$& 
  $e_1 e_1 = e_2$&   $e_1 e_3 = \af e_4$&    $e_2 e_2 = e_4$& \multicolumn{2}{l}{$e_3 e_1 =(\af+1) e_4$}&\\

  
  \hline
$\cd {4}{27}$&$:$& 
  $e_1 e_1 = e_2$&   $e_2 e_1 = e_4$& $e_2 e_3 = e_4$&  $e_3 e_2 = e_4$&   $e_3 e_3 = e_4$&\\

  \hline
$\cd {4}{28}(\af)$&$:$& 
  $e_1 e_1 = e_2$&   $e_2 e_1 = e_4$&    $e_2 e_2 = e_4$\\
  $\af\ne 1$&& $e_2 e_3 =   e_4$&   $e_3 e_2 =  e_4$&   $e_3 e_3 =\af  e_4$&\\

  \hline
$\cd {4}{29}$&$:$& 
  $e_1 e_1 = e_2$&    $e_2 e_3 =-\frac{1}{2} e_4$&   $e_3 e_2 =\frac{1}{2} e_4$&   $e_3 e_3 = e_4$&\\

  \hline
$\cd {4}{30}$&$:$& 
  $e_1 e_1 = e_2$&   $e_2 e_1 = e_4$& $e_2 e_3 = e_4$&   $e_3 e_2 = e_4$&\\

  \hline
$\cd {4}{31}$&$:$& 
  $e_1 e_1 = e_2$&  $e_2 e_3 = e_4$&     $e_3 e_1 = e_4$&
$e_3 e_2 = e_4$&\\

  \hline
$\cd {4}{32}(\af)$&$:$& 
  $e_1 e_1 = e_2$&   $e_2 e_3 = \af e_4$&   \multicolumn{2}{l}{$e_3 e_2 = (\af+1) e_4$}& \\

  \hline
$\cd {4}{33}$&$:$& 
  $e_1 e_1 = e_2$&   $e_2 e_1 = e_4$&   $e_2 e_2 = e_4$& $e_3 e_3 = e_4$& \\

 \hline
$\cd {4}{34}$&$:$& 
  $e_1 e_1 = e_2$&    $e_2 e_2 = e_4$& $e_3 e_3 = e_4$& \\

 \hline
$\cd {4}{35}$&$:$& 
  $e_1 e_1 = e_2$&    $e_2 e_2 = e_4$&  $e_3 e_1 = e_4$& $e_3 e_3 = e_4$& \\

 \hline
$\cd {4}{36}(\af)$&$:$& 
  $e_1 e_1 = e_2$&     $e_2 e_2 = e_4$&   $e_3 e_2= e_4$& $e_3 e_3 =\af  e_4$&\\


 \hline
$\cd {4}{37}$&$:$& 
  $e_1 e_1= e_2$&    $e_1 e_2= e_4$&   $e_2 e_1= e_4$&    $e_3 e_3= e_4$\\


  \hline
$\cd {4}{38}$&$:$& 
  $e_1 e_1= e_2$&   $e_2 e_3= e_4$&   $e_3 e_2= e_4$&\\

  
\hline
$\cd 4{39}$ &:&  
$e_1 e_1 = e_3+e_4$ & $e_1e_2=\frac i2 e_4$ & $e_1e_3=e_4$ & $e_2e_1=\frac i2 e_4$\\
&& $ e_2 e_2=e_3$ & $e_2e_3=-2ie_4$ & $e_3e_1=2e_4$ & $e_3e_2=-ie_4$\\

\hline
$\cd 4{40}$ &:&  
$e_1 e_1 = e_3+e_4$ & $e_1e_2=\frac i2 e_4$ & $e_1e_3=-\frac 12e_4$ & $e_2e_1=\frac i2 e_4$\\
&& $ e_2 e_2=e_3$ & $e_2e_3=-\frac i2e_4$ & $e_3e_1=\frac 12e_4$ & $e_3e_2=\frac i2e_4$\\

\hline
$\cd 4{41}$ &:&  
$e_1 e_1 = e_3+e_4$ & $e_1e_2=-\frac i2 e_4$ & $e_1e_3=e_4$ & $e_2e_1=-\frac i2 e_4$\\
&& $ e_2 e_2=e_3$ & $e_2e_3=-2ie_4$ & $e_3e_1=2e_4$ & $e_3e_2=-ie_4$\\

\hline
 $\cd 4{42}$ &:&  
$e_1 e_1 = e_3+e_4$ & $e_1e_2=-\frac i2 e_4$ & $e_1e_3=-\frac 12e_4$ & $e_2e_1=-\frac i2 e_4$\\
&& $ e_2 e_2=e_3$ & $e_2e_3=-\frac i2e_4$ & $e_3e_1=\frac 12e_4$ & $e_3e_2=\frac i2e_4$\\

\hline
$\cd 4{43}(\af)$ &$:$&  
$e_1 e_1 = e_3 + e_4$     & $e_1 e_2 = \af e_4$          & $e_1 e_3 = -\frac 12 e_4$ \\
&&   $e_2 e_1 = \af e_4$      & $e_2 e_2 = e_3$           & $e_3 e_1 = \frac 12 e_4$\\

\hline
$\cd 4{44}(\af,\bt,\gm)$ &: &  
$e_1 e_1 = e_3 + \af e_4$   & $e_1 e_2 = \bt e_4$ & $e_2 e_1 = (\bt+\gm) e_4$ \\
&& $e_2 e_2 = e_3$ & $e_3 e_1 = e_4$             & $e_3 e_3 = e_4$\\

\hline
$\cd 4{45}$ &: &  
$e_1 e_1 = e_3 + 2i e_4$  & $e_1 e_2 = e_4$ & $e_2 e_1 = e_4$& $e_2 e_2 = e_3$\\
&&  $e_3 e_1 = e_4$           & $e_3 e_2 = i e_4$       & $e_3 e_3 = e_4$\\

\hline
$\cd 4{46}(\af)$ &: &
$e_1 e_1 = e_3 - 2i\af e_4$ & $e_1 e_2 = \af e_4$  & $e_2 e_1 = \af e_4$ & $e_2 e_2 = e_3$ \\
$\af\ne$ 0& & $e_3 e_1 = e_4$              & $e_3 e_2 = i e_4$       & $e_3 e_3 = e_4$\\

\hline
$\cd 4{47}(\af,\bt)$ &: & 
$e_1 e_1 = e_3 + e_4$  & $e_1 e_3 = \af e_4$  & $e_2 e_2 = e_3$\\ 
$\bt\ne 0$ && $e_2 e_3 = \bt e_4$  & $e_3 e_1 =(\af+1) e_4$        & $e_3 e_2 = \bt e_4$ \\

\hline		
$\cd 4{48}(\af)$ &: & 
 $e_1 e_1 = e_3 + \af e_4$ &  $e_2 e_1 = i\af e_4$ & $e_2 e_2 = e_3$  \\
$\af\ne$ 0&& $e_3 e_1 = e_4$            & $e_3 e_2 = i e_4$     & $e_3 e_3 = e_4$\\

\hline 
$\cd 4{49}(\af)$ &: & 
$e_1 e_1 = e_3 + \af e_4$ & $e_2 e_1 = -i\af e_4$ & $e_2 e_2 = e_3$\\
$\af\ne$ 0& & $e_3 e_1 = e_4$& $e_3 e_2 = i e_4$ & $e_3 e_3 = e_4$\\

\hline 
$\cd 4{50}(\af)$ &:    & 
$e_1 e_1 = e_3 + \af e_4$ & $e_2 e_1 = e_4$ & $e_2 e_2 = e_3$ & $e_3 e_3 = e_4$\\

\hline
$\cd 4{51}(\af)$ &: &
$e_1 e_1 = e_3 + \af e_4$ &   $e_2 e_2 = e_3$ & $e_3 e_1 = e_4$\\
&& $e_3 e_2 = i e_4$ & $e_3 e_3 = e_4$\\

\hline 
$\cd 4{52}$ &: &  
$e_1 e_1 = e_3 +  e_4$ &   $e_2 e_2 = e_3$ &   $e_3 e_3 = e_4$\\

\hline 
$  \cd 4{53}$ &: &  
$e_1 e_1 = e_3$ &  $e_1e_2=-\frac 12e_4$ & $e_1e_3=e_4$ & $e_2e_1=\frac 12e_4$\\
 && $e_2 e_2 = e_3$ & $e_2e_3=ie_4$ &  $e_3e_1=e_4$ & $e_3e_2=ie_4$ \\

\hline
 $\cd 4{54}(\af)$ &:&  
$e_1 e_1 = e_3$ & $e_1e_2=e_4$ & $e_1e_3=\af e_4$ & $e_2e_1=e_4$\\
&& $ e_2 e_2=e_3$ & $e_2e_3=-i(\af+1)e_4$ & $e_3e_1=(\af+1) e_4$ & $e_3e_2=-i\af e_4$\\

\hline
$\cd 4{55}(\af)$ &: &  
$e_1 e_1 = e_3$ & $e_1 e_2 = e_4$ & $e_1 e_3 = \af e_4$ \\
& & $e_2 e_1 = e_4$ & $e_2 e_2 = e_3$  & $e_3 e_1 = (\af+1) e_4$ \\
		
\hline
$\cd 4{56}$ &: &  
$e_1 e_1 = e_3$ & $e_1 e_2 = e_4$ & $e_2 e_1 = -e_4$ & $e_2 e_2 = e_3$ \\
& & $e_3 e_1 = e_4$ & $e_3 e_2 = i e_4$ & $e_3 e_3 = e_4$\\

\hline
$\cd 4{57}(\af,\bt)$ &: &  
$e_1 e_1 = e_3$ & $e_1 e_2 = \af e_4$ &   $e_2 e_1 = (\af+\bt)e_4$ & $e_2 e_2 = e_3$ \\
$\bt\not\in\{0,-2\af\}$& & $e_3 e_1 = e_4$          & $e_3 e_2 = i e_4$     & $e_3 e_3 = e_4$\\

\hline
$\cd 4{58}$ &: &  
$e_1 e_1 = e_3$ & $e_1 e_3 = e_4$  & $e_2 e_1 = e_4$ & $e_2 e_2 = e_3$\\
&& $e_2 e_3 = i e_4$  & $e_3 e_1 = e_4$ & $e_3 e_2 = i e_4$ \\

\hline 
$\cd 4{59}(\af,\bt)$ &: &   
$e_1 e_1 = e_3$ & $e_1 e_3 = \af e_4$ & $e_2 e_2 = e_3$\\
 $\bt\ne 0$ && $e_2 e_3 = \bt e_4$ & $e_3 e_1 = (\af+1) e_4$  & $e_3 e_2 = \bt e_4$ \\

\hline		
$\cd 4{60}$ &: & 
$e_1 e_1 = e_3$ & $e_1 e_3 = i e_4$  & $e_2 e_2 = e_3$\\
&& $e_2 e_3 = e_4$ & $e_3 e_1 = (i+1) e_4$ & $e_3 e_2 = (i+1) e_4$  \\

\hline
$\cd 4{61}(\af)$ &: & 
 $e_1 e_1 = e_3$ & $e_1 e_3 = -i\af e_4$ & $e_2 e_2 = e_3$\\
 && $e_2 e_3 = \af e_4$ & $e_3 e_1 = (1-i\af) e_4$ & $e_3 e_2 = (\af+i) e_4$  \\

\hline
$\cd 4{62}$ &:&  
$e_1 e_1 = e_3$  & $e_1e_3=e_4$ & $e_2 e_2=e_3$\\
&& $e_2e_3=-2ie_4$ & $e_3e_1=2e_4$ & $e_3e_2=-ie_4$\\

\hline
$\cd 4{63}$ &:&  
$e_1 e_1 = e_3$ & $e_1e_3=-\frac 12e_4$ & $ e_2 e_2=e_3$\\
&& $e_2e_3=-\frac i2e_4$ & $e_3e_1=\frac 12e_4$ & $e_3e_2=\frac i2e_4$\\

\hline 
$\cd 4{64}(\af)$ &: & 
$e_1 e_1 = e_3$ & $e_1 e_3 = \af e_4$ & $e_2 e_2 = e_3$  & $e_3 e_1 = (\af+1) e_4$ 		\\

\hline
$\cd 4{65}(\af)$ &: &  
$e_1 e_1 = e_3$ &   $e_1 e_3 = \af e_4$ & $e_2 e_2 = e_3$ \\
$\af\ne 0$&&  $e_3 e_1 = (\af+1) e_4$  & $e_3 e_2 = i e_4$ \\
\hline
$\cd 4{66}$ &: &  
$e_1 e_1 = e_3$& $e_2 e_1 = e_4$ & $e_2 e_2 = e_3$ \\
&& $e_2 e_3 = e_4$ & $e_3 e_2 = e_4$ \\

\hline 
$\cd 4{67}$ &: & 
$e_1 e_1 = e_3$ &  $e_2 e_2 = e_3$ &  $e_3 e_3 = e_4$\\

\hline
$\cd 4{68}$ &: & 
$e_1 e_1 = e_3+e_4$ & $e_1 e_3 = i e_4$  &$e_2 e_2 = e_3$ &\\
&&$e_2 e_3 = e_4$&$e_3 e_1 =   ie_4$ & $e_3 e_2 = e_4$  \\

\hline
$\cd 4{69}$ &: & 
$e_1 e_1 = e_3$ & $e_1 e_3 =  ie_4$  &$e_2 e_2 = e_3$ &\\
&&$e_2 e_3 = e_4$&$e_3 e_1 =   ie_4$ & $e_3 e_2 = e_4$  \\

\hline
$\cd 4{70}$ &: & 
$e_1 e_1 = e_3$ & $e_1 e_3 =    e_4$  &$e_2 e_2 = e_3$ & $e_3 e_1 = e_4$ &   \\
\hline
		$\cd 4{71}$ &: & $e_1 e_1 = e_4$ & $e_1 e_2 = e_3$  & $e_1 e_3 = e_4$ & $e_2 e_1 = -e_3$  & $e_3 e_1 = -e_4$\\
		\hline
		$\cd 4{72}$ &: & $e_1 e_1 = e_4$ & $e_1 e_2 = e_3$ & $e_1 e_3 = e_4$\\
		&& $e_2 e_1 = -e_3$  & $e_2 e_2 = e_4$ & $e_3 e_1 = -e_4$\\
		\hline
		$\cd 4{73}$ &: & $e_1 e_2 = e_3 + e_4$ & $e_1 e_3 = e_4$ & $e_2 e_1 = -e_3$  & $e_3 e_1 = -e_4$\\
		\hline
		$\cd 4{74}(\af)$ &: & $e_1 e_2 = e_3$ & $e_1 e_3 = (\af+1)e_4$ & $e_2 e_1 = -e_3$ & $e_3 e_1 = -\af e_4$\\
		\hline
		$\cd 4{75}(\af)$ &: & $e_1 e_2 = e_3$ & $e_1 e_3 = (\af+1)e_4$ & $e_2 e_1 = -e_3$ & $e_2 e_2 = e_4$ & $e_3 e_1 = -\af e_4$\\        \hline
		$\cd 4{76}$ &: & $e_1 e_2 = e_3$ & $e_1 e_3 = e_4$ & $e_2 e_1 = -e_3$ & $e_2 e_2 = e_4$ & $e_3 e_1 = -e_4$\\        
		\hline
		$\cd 4{77}$ &: & $e_1 e_2 = e_3$ & $e_1 e_3 = e_4$ & $e_2 e_1 = -e_3$ & $e_2 e_3 = e_4$ & $e_3 e_2 = -e_4$\\
		\hline
		$\cd 4{78}$ &: & $e_1 e_2 = e_3$ & $e_1 e_3 = e_4$ & $e_2 e_1 = -e_3$\\
		&& $e_2 e_2 = e_4$ & $e_2 e_3 = e_4$ & $e_3 e_2 = -e_4$\\
		\hline
		$\cd 4{79}(\af)$ &: & $e_1 e_2 = e_3+\af e_4$ & $e_1 e_3 = e_4$ & $e_2 e_1 = -e_3$ & $e_2 e_3 = e_4$ & $e_3 e_3 = e_4$\\
		\hline
		$\cd 4{80}$ &: & $e_1 e_2 = e_3+e_4$ & $e_1 e_3 = e_4$ & $e_2 e_1 = -e_3$ & $e_3 e_3 = e_4$\\
		\hline
		$\cd 4{81}$ &: & $e_1 e_2 = e_3+e_4$ & $e_2 e_1 = -e_3$ & $e_3 e_3 = e_4$\\
		\hline
		$\cd 4{82}$ &: & $e_1 e_2 = e_3$ & $e_1 e_3 = e_4$ & $e_2 e_1 = -e_3$ & $e_2 e_2 = e_4$ & $e_3 e_3 = e_4$\\
		\hline
		$\cd 4{83}( \af\ne 0)$ &: & $e_1 e_2 = e_3$ & $e_2 e_1 = -e_3$ & $e_2 e_2 = \af e_4$ & $e_2 e_3 = e_4$ & $e_3 e_3 = e_4$\\
		\hline
		$\cd 4{84}$ &: & $e_1 e_2 = e_3$ & $e_2 e_1 = -e_3$ & $e_2 e_2 = e_4$ & $e_3 e_3 = e_4$\\
		\hline
		$\cd 4{85}$ &: & $e_1 e_2 = e_3$ & $e_2 e_1 = -e_3$ & $e_3 e_3 = e_4$\\
		\hline
		$\cd 4{86}$ &: & $e_1 e_2 = e_3$ & $e_2 e_1 = -e_3$ & $e_2 e_3 = e_4$  &  $e_3 e_2 = -e_4$\\
\hline		
		
$\cd {4}{87}(\lambda)$&$:$& 
$e_1 e_1 = \lambda e_3+(2 \Theta-1)e_4$& $e_1 e_2=e_4$ & $e_1e_3=e_4$& 
\multicolumn{2}{l}{$e_2 e_1=e_3-(1- \Theta)^2 \lb^{-1}e_4$}\\
$\lb \neq 0, \frac{1}{4}$&& $e_2 e_2=e_3$& $e_2e_3=\Theta\lb^{-1}e_4$ 
&$e_3e_3=e_4$\\

\hline$\cd {4}{88}(\lambda)$&$:$& 
$e_1 e_1 = \lambda e_3+(1-2 \Theta)e_4$& $e_1 e_2=e_4$ & $e_1e_3=e_4$& 
\multicolumn{2}{l}{$e_2 e_1=e_3- \Theta^2 \lb^{-1}e_4$}\\
$\lb \neq 0, \frac{1}{4}$&& $e_2 e_2=e_3$& $e_2e_3=\Theta\lb^{-1}e_4$ 
&$e_3e_3=e_4$\\

\hline$\cd {4}{89}(\lambda)$&$:$& 
$e_1 e_1 = \lambda e_3+(2 \Theta-1)e_4$& $e_1 e_2=e_4$ & $e_1e_3=e_4$& 
\multicolumn{2}{l}{$e_2 e_1=e_3-(1- \Theta)^2 \lb^{-1}e_4$}\\
$\lb \neq 0, \frac{1}{4}$&& $e_2 e_2=e_3$& $e_2e_3=(1-\Theta)\lb^{-1}e_4$ 
&$e_3e_3=e_4$\\

\hline$\cd {4}{90}(\lambda)$&$:$& 
$e_1 e_1 = \lambda e_3+(1-2 \Theta)e_4$& $e_1 e_2=e_4$ & $e_1e_3=e_4$& 
\multicolumn{2}{l}{$e_2 e_1=e_3- \Theta^2 \lb^{-1}e_4$}\\
$\lb \neq 0, \frac{1}{4}$&& $e_2 e_2=e_3$& $e_2e_3=(1-\Theta)\lb^{-1}e_4$ 
&$e_3e_3=e_4$\\

\hline$\cd {4}{91}(\lambda, \af)$&$:$& 
$e_1 e_1 = \lambda e_3+(2 \Theta-1)e_4$& $e_1 e_2=e_4$ & $e_1e_3=\af e_4$& \\
$\lb \neq 0, \frac{1}{4}$&& $e_2 e_1=e_3- (1-\Theta)^2 \lb^{-1}e_4$&
$e_2 e_2=e_3$&  $e_3e_3=e_4$\\

\hline$\cd {4}{92}(\lambda, \af)$&$:$& 
$e_1 e_1 = \lambda e_3+(1-2 \Theta)e_4$& $e_1 e_2=e_4$ & $e_1e_3=\af e_4$& \\
$\lb \neq 0, \frac{1}{4}$&& $e_2 e_1=e_3-  \Theta^2 \lb^{-1}e_4$&
$e_2 e_2=e_3$&  $e_3e_3=e_4$\\

\hline$\cd {4}{93}( \af)$&$:$& 
$e_1 e_1 = e_4$ & $e_1 e_2=e_4$ &$e_1e_3=\af e_4$ & $e_2e_1=e_3+e_4$\\
&&$ e_2 e_2=e_3$ &$e_2e_3=\af e_4$ & $e_3e_3=e_4$\\

\hline$\cd {4}{94}(\af, \bt)$&$:$& 
$e_1 e_1 = e_4$ & $e_1e_2=e_4$& $e_1e_3=\af e_4$&\\
$\af\neq0$&&$e_2 e_1=e_3+\beta e_4$   & $e_2 e_2=e_3$   &$e_3e_3=e_4$\\

\hline$\cd {4}{95}(\af)$&$:$& 
$e_1 e_1 =  e_4$ & $e_1e_2=e_4$ &$e_1e_3=\af e_4$ & $e_2 e_1=e_3$\\
&& $e_2 e_2=e_3$ & $e_2e_3=\af e_4$&$e_3e_3=e_4$\\

\hline$\cd {4}{96}(\af)$&$:$& 
$e_1 e_1 =  e_4$ & $e_1e_2=e_4$ & $e_2 e_1=e_3+\af e_4$\\
&& $e_2 e_2=e_3$ & $e_2e_3= e_4$&$e_3e_3=e_4$\\

\hline$\cd {4}{97}(\lambda)$&$:$& 
$e_1 e_1 = \lambda e_3$ & $e_1e_2=e_4$& $e_1e_3=\Theta e_4 $&$e_2 e_1=e_3-e_4$ \\
&& $e_2 e_2=e_3$ &$e_2e_3=e_4$ &$e_3e_3=e_4$\\

\hline$\cd {4}{98}(\lambda)$&$:$& 
$e_1 e_1 = \lambda e_3$ & $e_1e_2=e_4$& $e_1e_3=(1-\Theta) e_4 $&$e_2 e_1=e_3-e_4$ \\
$\lb \neq \frac{1}{4}$ && $e_2 e_2=e_3$ &$e_2e_3=e_4$ &$e_3e_3=e_4$\\

\hline$\cd {4}{99}(\af)$&$:$& 
$e_1e_2=e_4$& $e_1e_3=e_4$& $e_2 e_1=e_3-e_4$\\ 
$\af\neq1$ && $e_2 e_2=e_3$ &$e_2e_3=\af e_4$&$e_3e_3=e_4$\\

\hline$\cd {4}{100}(\af)$&$:$& 
$e_1 e_1 = \frac{1}{4} e_3$ & $e_1e_2=e_4$& $e_1e_3=\af e_4$& $e_2 e_1=e_3-e_4$ \\
$\af\notin\{0, \frac{1}{2}\}$&&  $e_2 e_2=e_3$ &$e_2e_3=2 \af e_4$ &$e_3e_3=e_4$\\

 \hline$\cd {4}{101}(\af, \bt)$&$:$& 
 $e_1e_2=e_4$& $e_1e_3=\af e_4$& $e_2 e_1=e_3$  \\
 && $e_2 e_2=e_3$&$e_2e_3=\bt e_4$ &$e_3e_3=e_4$\\
 
 \hline$\cd {4}{102}(\lb, \af)$&$:$& 
 $e_1 e_1 = \lambda e_3$ & $e_1e_2=e_4$& $e_2 e_1=e_3-e_4$ \\
 $\lb\neq0$&& $e_2 e_2=e_3$&$ e_2e_3=\af e_4$  &$e_3e_3=e_4$\\

\hline$\cd {4}{103}$&$:$& 
$e_1 e_2 = e_4$ & $e_2 e_1=e_3-e_4$  & $e_2 e_2=e_3$ &$e_3e_3=e_4$\\
 
 \hline$\cd {4}{104}$&$:$& 
 $e_1 e_3 =  e_4$ & $e_2 e_1=e_3+e_4$  & $e_2 e_2=e_3$ &$e_2e_3=e_4$&$e_3e_3=e_4$\\
 
 \hline$\cd {4}{105}(\lambda, \af,\bt)$&$:$& 
 $e_1 e_1 = \lambda e_3$ & $e_1e_3=e_4$&$e_2 e_1=e_3+\af e_4$  \\
 $  \lb\ne 0, \af\ne 0$&& $e_2 e_2=e_3$ & $e_2e_3=\bt e_4$&$e_3e_3=e_4$\\
 
 \hline$\cd {4}{106}(\af)$&$:$& $e_1 e_3 = e_4$ & $e_2 e_1=e_3+\af e_4$  & $e_2 e_2=e_3$ &$e_3e_3=e_4$\\
 
 \hline$\cd {4}{107}(\lambda)$&$:$& 
 $e_1 e_1 = \lambda e_3$ & $e_1e_3=\Theta e_4$& $e_2 e_1=e_3$  \\
 && $e_2 e_2=e_3$ & $e_2e_3=e_4$&$e_3e_3=e_4$\\
 
 \hline$\cd {4}{108}(\lambda)$&$:$& 
 $e_1 e_1 = \lambda e_3$ & $e_1e_3=(1-\Theta) e_4$& $e_2 e_1=e_3$  \\
 $\lb \not\in \{0, \frac{1}{4}\}$&& $e_2 e_2=e_3$ & $e_2e_3=e_4$&$e_3e_3=e_4$\\
 
 \hline$\cd {4}{109}(\lambda,\af)$&$:$& $e_1 e_1 = \lambda e_3$ & $e_2 e_1=e_3+e_4$  & $e_2 e_2=e_3$ &$e_2e_3=\alpha e_4$&$e_3e_3=e_4$\\

  \hline$\cd {4}{110}(\lambda)$&$:$& $e_1 e_1 = \lambda e_3$ & $e_2 e_1=e_3$  & $e_2 e_2=e_3$ &$e_2e_3= e_4$&$e_3e_3=e_4$\\
 
   \hline$\cd {4}{111}(\lambda)$&$:$& $e_1 e_1 = \lambda e_3$ & $e_2 e_1=e_3$  & $e_2 e_2=e_3$ &$e_3e_3=e_4$\\
 
\hline$\cd {4}{112}(\lambda, \af, \bt, \gamma)$&$:$& 
$e_1 e_1 = \lambda e_3+e_4$ & $e_1e_3=\af e_4$ & $e_2 e_1=e_3+\bt e_4$  \\
&&$e_2 e_2=e_3$& $e_2e_3=\gamma e_4$&$e_3e_3=e_4$\\
\end{longtable}}

\section{The geometric classification of complex $4$-dimensional nilpotent $\mathfrak{CD}$-algebras}
\begin{theorem}\label{main-geo}
The variety of complex $4$-dimensional nilpotent terminal algebras has $2$ irreducible components: one of dimension $18$ determined by the family of algebras $\cd 4{112}(\lb,\af,\bt,\gm)$ and one of dimension $15$ determined by the family of algebras $\cd 4{12}(\lb,\af)$. 
\end{theorem}

\begin{proof}[{\bf Proof}]
Thanks to \cite{kppv, kppvcor} the algebras 
 $\mathfrak{N}_2(\alpha)$ and $\mathfrak{N}_3(\alpha)$  
 determine the irreducible components in the variety of complex $4$-dimensional $2$-step nilpotent algebras (which are trivial $\mathfrak{CD}$-algebras in our terminology),
 and the orbit closure of $\D{4}{01}(\lambda,\alpha,\beta)$ contains all terminal (Leibniz and non-Leibniz) extensions of $\cd {3*}{04}(\lambda)$ (see \cite{kppvcor,kkp19geo}), where
\begin{longtable}{lllllllllll}

        $\D{4}{01}(\lambda,\alpha,\beta)$&$:$& $e_1 e_1 = \lambda e_3 + e_4$ & $e_1 e_3 = \alpha e_4$ & $e_2 e_1=e_3$ 
        & $e_2 e_2 = e_3$ & $e_2 e_3 = \beta e_4$ & $e_3e_1 = e_4$\\
        
        $\mathfrak{N}_2(\alpha)$&$:$& $e_1e_1 = e_3$ &$e_1e_2 = e_4$  &$e_2e_1 = -\alpha e_3$ &$e_2e_2 = -e_4$ \\

        $\mathfrak{N}_3(\alpha)$&$:$& $e_1e_1 = e_4$ &$e_1e_2 = \alpha e_4$  &$e_2e_1 = -\alpha e_4$ &$e_2e_2 = e_4$  &$e_3e_3 = e_4$

        \end{longtable}
    
Each algebra of the family $\cd 4{112}(\lb,\af,\bt,\gm)$ has $2$-dimensional square, while 
the square of $\cd 4{12}(\lb,\af)$ is of dimension $3$.
Hence, $\cd 4{112}(\lb,\af,\bt,\gm)\not\to\cd 4{12}(\lb,\af)$ by Lemma~\ref{main}. We also have $\cd 4{12}(\lb,\af)\not\to\cd 4{112}(\lb,\af,\bt,\gm)$, because $\overline{\bigcup\orb(\cd 4{112}(\lb,\af,\bt,\gm))}$ has dimension $18$, while $\overline{\bigcup\orb(\cd 4{12}(\lb,\af))}$ is of dimension $15$.
The rest of the algebras degenerate either from $\cd 4{12}(\lb,\af)$, or from $\cd 4{112}(\lb,\af,\bt,\gm)$, as shown below. In some of the degenerations we use the short notation $\Theta=\frac{1+\sqrt{1-4\Lambda}}2$ and $\Psi=\frac{1-\sqrt{1-4\Lambda}}2=1-\Theta$.
        
{\tiny
\begin{longtable}{lll| llll}

 \hline
 $\cd 4{112}$ &$\to$ &  $\D{4}{01}(\Lambda, \A, \B)$ & 
 $\lambda= \Lambda $ &$\af=(\A-1)\Xi $ & $\bt=-\B\Xi^2  $ & $\gamma= \B\Xi  $\\
  \multicolumn{3}{l|}{$\Xi=(t-\A)^{-\frac 12}$}&
 $E_1^t=  t\Xi (e_1+\Xi e_3)$ &$E_2^t=t\Xi e_2$ & $E_3^t= t^2\Xi^2 e_3 $ & $E_4^t= t^3\Xi^4 e_4$\\

 
 \hline
 $\cd 4{12}$ &$\to$ & $\mathfrak{N}_2( \A)$ & 
 $\lambda=t^{-1} $ &$\af=t^{-1} $ \\
 \multicolumn{3}{l|}{$\Xi=(\A t(t - 1)(t^2 + 3t + 1))^{-1}$}& 
 \multicolumn{2}{l}{$E_1^t= t e_1-t\Xi(t^2 - \A t - 1)e_2$} &&\\
&&&\multicolumn{4}{l}{ $E_2^t= e_2 + ((\A + 1)t + \A - 1) e_3 + \frac{\Xi^2}{t^2}(\A t^4 + 2(\A-1)(t^3+\frac 1t) - 3(2\A+1)t^2 + (\A^2  - 5\A + 10)t + 3(2\A-1)) e_4
$}\\
&& & \multicolumn{2}{l}{$E_3^t=- \A^{-1}e_3 -\Xi\A^{-1}(\A t^2 - 3t - 2\A + 3) e_4$} & $E_4^t= -t^{-1} e_4$\\

 \hline
 $\cd 4{112}$ &$\to$ & $\mathfrak{N}_3( \A)$ & 
 $\lambda=1 $ &$\af=-t $ & $\bt= \A^2-t^2 $ & $\gamma=t  $\\
 &&& 
 $E_1^t= -\A^2t  (e_1-e_2+te_3)$ &
$E_2^t= \A t e_2$
& $E_3^t=\A^2t  (e_3 -\A^2e_4)$ & $E_4^t= \A^4t^2  e_4$\\

 \hline
 $\cd 4{12}$ &$\to$ &  $\cd 4{01} $ & 
 $\lambda= 2$ &$\af= 3t^{-4} $  \\
 &&& $E_1^t= t  e_1$ &$E_2^t= t^2e_2$ & $E_3^t=  3e_4 $ & $E_4^t=  e_3$\\

 \hline
 $\cd 4{12}$ &$\to$ &  $\cd 4{02} $ & 
 $\lambda=2 $ &$\af=t^{-4} $  \\
\multicolumn{3}{l|}{$\Xi=(9t^4+1)^{-1}$}& 
$E_1^t=  t (e_1-t \Xi) e_2$ &
\multicolumn{2}{l}{ $E_2^t= t^2e_2-3 t^3 \Xi e_3+\Xi^2 e_4$}\\&& & 
 $E_3^t=  e_4 $ & $E_4^t=  e_3-\Xi t^{-3}e_4$\\

 \hline
 $\cd 4{12}$ &$\to$ &  $\cd 4{03}$ & 
 $\lambda=-2 $ &$\af=0 $  \\
 &&& $E_1^t= te_1-\frac{1}{3t^{2}}e_2$ &$E_2^t= t^2e_2-\frac{1}{3t^1}e_3$  
 & $E_3^t=  e_4$ & $E_4^t= t^2 e_3+\frac{4}{3t}e_4 $\\

 \hline
 $\cd 4{12}$ &$\to$ &  $\cd 4{04}(\Lambda)$ & 
 $\lambda= \Lambda$ &$\af=0 $  \\
 &&& $E_1^t=  t e_1 $ &$E_2^t= t^2 e_2 $ & $E_3^t= e_4$ & $E_4^t=  e_3-t^{-3}e_4 $  
  \\

 \hline
 $\cd 4{12}$ &$\to$ &  $\cd 4{05}$ & 
 $\lambda=t^{-1}$ &$\af= t^{-3} $\\
 &&& $E_1^t=   te_1$ &$E_2^t=t^2 e_2$   & $E_3^t= t^2 e_4$ & $E_4^t=  te_3 $\\

 \hline
 $\cd 4{12}$ &$\to$ &  $\cd 4{06}$ & 
 $\lambda= t^{-1}$ &\multicolumn{2}{l}{$\af= -\Xi t^{-2}(t^3+t^2-t-2)$}  \\
 \multicolumn{3}{l|}{$\Xi=(t-1)^{-1}$}& 
 \multicolumn{1}{l}{$E_1^t=   te_1-e_2$} &
 \multicolumn{2}{l}{$E_2^t=t^2 e_2-(t+1)e_3 -\Xi t^{-2}(t^3+t^2-t-2) e_3$} \\
 &&& \multicolumn{1}{l}{$E_3^t=t^2  e_3+\Xi(3 t^2-4)e_4 $} & $E_4^t=  e_4$\\

 \hline
 $\cd 4{12}$ &$\to$ &  $\cd 4{07}(\Lambda)$ & 
 $\lambda=\Lambda $ &$\af=(\Lambda-2)t^{-1} $\\
 &&& $E_1^t=  t e_1$ &$E_2^t=t^2 e_2$ & $E_3^t=  t^3 (\Lambda-2)e_4 $ & $E_4^t= t^3 e_3$\\

 \hline
 $\cd 4{112}$ &$\to$ &  $\cd 4{08}( \A)$ & 
 $\lambda=t^{-1} $ &$\af= \A(1-2t)t^{-2} $ \\
 &&& $E_1^t= t  e_1$ &$E_2^t= t^2e_2$ & $E_3^t= t^2 e_3 $ & $E_4^t=t^2(1-2t)  e_4$\\

 \hline
 $\cd 4{12}$ &$\to$ &  $\cd 4{09}$ & 
 $\lambda= t^{-1} $ &$\af=(2t-1)t^{-2} $ \\
 \multicolumn{3}{l|}{$\Xi=((t-1)(t+4))^{-1}$}& 
 \multicolumn{2}{l}{$E_1^t=  t e_1+\Xi(2t-1)(t e_2+\Xi (2t-1)^2(t+1)^{-1}e_3)$} &
\multicolumn{2}{l}{$E_2^t= t^2e_2+\Xi t (t+1) (2 t-1)e_3$}\\ 
&&& 
\multicolumn{2}{l}{$E_3^t=t^2  e_3+\Xi t (2 t-1) (t^2+t-3)e_4 $} & $E_4^t= t^2(1-2 t) e_4$\\

 \hline
 $\cd 4{12}$ &$\to$ &  $\cd 4{10}( \A)$ & 
 $\lambda=1+t $ &$\af=\A $ & \\
 \multicolumn{3}{l|}{$\Xi=(t ((t+2)^2+\A))^{-1}$}& 
 \multicolumn{2}{l}{$E_1^t=  e_1+\Xi(t^3-1)(e_2+\A\Xi(t+2)^{-1}(t^3-1) e_3)$} &
 \multicolumn{2}{l}{$E_2^t= e_2+\Xi (t+2) (t^3-1)e_3$}\\ 
 &&& 
 \multicolumn{2}{l}{$E_3^t=  e_3+\Xi(1-t) ((t+2)(2t^2+1)-\A(t+1))e_4 $} & $E_4^t=  (1-t)e_4$\\
 
  \hline
 $\cd 4{12}$ &$\to$ &  $\cd 4{11}(\Lambda )$ & 
 $\lambda= \Lambda$ &$\af=t-(\Lambda+1)^2 $ & \\
\multicolumn{3}{l|}{$\Xi=\Lambda-1$}& 
 \multicolumn{1}{l}{$E_1^t=  t\Xi(e_1+(t+1)e_2)$} &
 \multicolumn{3}{l}{$E_2^t=t^2\Xi^2(e_2-(t+1) ((\Lambda+1) e_3-(t+1)(t-(\Lambda+1)^2))e_4)$}\\ &&& 
  \multicolumn{2}{l}{$E_3^t= t^3\Xi^3(e_3 - (t+1)(t-3 (\Lambda+1))e_4)$} & $E_4^t= t^4 \Xi^4 e_4$\\

 \hline
 $\cd 4{112}$ &$\to$ &  $\cd 4{13}(\A)$ & 
 \multicolumn{4}{l}{$\Xi=\sqrt{(\A+\A^2)^3 (t^3+(2\A+1)(t^2+\A^2)+\A^4)}$}\\
 &&&
  \multicolumn{2}{l}{$\lambda=\frac{\A^2 (1+\A)^2}{t (\A+t)^2} $} &
   \multicolumn{2}{l}{$\af=-\frac{\A (1+\A) \sqrt{t} (\A+t)}{\Xi} $}\\
   &&& 
    \multicolumn{2}{l}{$\bt= \frac{t (\A+t) (1+\A+t)}{\A^2+2 \A^3+\A^4+t^2+2 \A t^2+t^3} $} & $\gamma= -\frac{\A (1+\A) \sqrt{t} (\A+t)}{\Xi}  $\\&&
 
 & \multicolumn{2}{l}{$E_1^t=- \frac{\sqrt{t^5}  (\A+t)^2}{\Xi}   \left(e_1-e_2\right)-
 \frac{t^3 (\A(t+\A)(1+\A))^2}{\Xi^2}e_3$} \\&&&
 \multicolumn{2}{l}{$E_2^t= \frac{t^4 \A^2(t+\A)^2(1+\A)^2}{\Xi^2} \left(e_3 + \frac{t\A^3 (\A+t)^2  (1+\A)^3}{\Xi^2}e_4\right)$} \\&& & 
 \multicolumn{2}{l}{$E_3^t=  \frac{\A (1+\A) \sqrt{t^5} (\A+t)}{\Xi}\left(e_2
 -\sqrt t (\A+t) \A^2(1+\A)e_3\right)$} & $E_4^t=-\frac{t^7 (\A+t)^4}{\Xi^4}  e_4$\\

 \hline
 $\cd 4{112}$ &$\to$ &  $\cd 4{14}( \A,\B)$ & 
 \multicolumn{2}{l}{$\Xi=\sqrt{(\B-t) (\B^2-\B t+(\A-1) (\A-t) t)}$}\\
&&& $\lambda=\frac{(\B-t)^2}{(\A-t)^2 t} $ &$\af=\frac{i (1-t) \sqrt{t} (t-\A)}{\Xi} $ & 
$\bt= \frac{(\A-1) (\A-t)(\B-t) t}{\Xi} $ & $\gamma= \frac{i (1-t) \sqrt{t} (t-\A)}{\Xi}  $\\
 &&& 
 \multicolumn{2}{l}{$E_1^t=(\A-t)^2 \sqrt{t^3}\left(\frac{i}{(\B-t)\Xi }\left(e_1-e_2\right) -\frac{\sqrt t}{\Xi^2}e_3\right)$} &
 \multicolumn{2}{l}{$E_2^t= -\frac{(\A-t)^2 t^2}{\Xi^2} \left(e_3+ \frac{(\A-t)^2 t (\B-t)}{\Xi^2}e_4\right)$} \\&& & 
\multicolumn{2}{l}{$E_3^t= \frac{\sqrt{t^3} (t-\A)}{\Xi} \left(ie_2-\frac{(t-\A) \sqrt{t^3}}{\Xi}e_3\right) $} & 
    $E_4^t= \frac{(\A-t)^4 t^4}{\Xi^4} e_4$\\

 \hline
 $\cd 4{112}$ &$\to$ &  $\cd 4{15}(\A)$ & 
 $\lambda=\frac{1}{t} $ &$\af= \frac{\A-t}{\A \Xi} $ & $\bt=-\frac{t}{\Xi^2}  $ & $\gamma=- \frac{t}{\A\Xi}  $\\
 \multicolumn{3}{l|}{$\Xi= i\sqrt{\A+\A^2+t}$}& 
 $E_1^t=  -\frac{t^2}{\Xi}\left(e_1-e_2\right)+\frac{(1+\A) t^2}{\Xi^2} e_3$ & $E_2^t= \frac{t^3}{\Xi^2}e_3$
 & 
 $E_3^t=  \frac{t^2}{\Xi}e_2+\frac{t^3}{\A \Xi^2}e_3 $ & $E_4^t= \frac{t^5}{\Xi^4}  e_4$\\

 \hline
 $\cd 4{112}$ &$\to$ &  $\cd 4{16}$ & 
 $\lambda= 2+\frac{1}{4 t}+4 t $ &$\af=\frac{4 \sqrt{t} (1+4 t)}{\Xi}$ & $\bt= \frac{4 t (1+4 t)^3}{\Xi^2} $ & $\gamma= \frac{4 \sqrt{t} (1+4 t)}{\Xi}  $\\
 \multicolumn{3}{l|}{$\Xi =i\sqrt{1+20 t+176 t^2}$}& \multicolumn{2}{l}{$E_1^t=-\frac{8 \sqrt{t^5}}{ \Xi }\left(e_1-e_2\right)+\frac{16 t^3 (1+4 t)^2}{\Xi^2}e_3$} &
 \multicolumn{2}{l}{$E_2^t= \frac{16 t^4 (1+4 t)^2}{\Xi^2}\left( e_3-\frac{4 t (1+4 t)^3}{\Xi^2}e_4\right)$}\\ 
 &&& 
 \multicolumn{2}{l}{$E_3^t= \frac{4 \sqrt{t^5} (1+4 t)}{\Xi} \left(e_2 - \frac{2\sqrt t (1+4 t)}{\Xi}e_3\right)$} & $E_4^t= \frac{256 t^7 (1+4 t)^4}{\Xi^4} e_4$\\

  \hline
 $\cd 4{112}$ &$\to$ &  $\cd 4{17}(\A)$ & 
 $\lambda= \frac{(1+\A)^2}{t}$ &$\af=\frac{\sqrt{t}}{\sqrt{\A(1+\A)^3}} $ & $\bt=\frac{t}{\A+\A^2}  $ & $\gamma= \frac{\sqrt{t}}{\sqrt{\A(1+\A)^3}} $\\
 &&& 
 \multicolumn{2}{l}{$E_1^t=-\frac{\sqrt{t^5}}{\sqrt{\A(1+\A)^5}}(e_1-e_2)+
 \frac{t^3}{\A(1+\A)^3}e_3 $} &
 \multicolumn{2}{l}{$E_2^t= \frac{t^4}{\A(1+\A)^3}\left(  e_3 +\frac{t}{(1+\A)^2}e_4\right) $} \\
 &&& \multicolumn{2}{l}{$E_3^t= \frac{\sqrt{t^5}}{\sqrt{\A(1+\A)^3}} e_2 +
 \frac{t^3}{(1+\A)^3}e_3 $} & $E_4^t= \frac{t^7}{\A^2(1+\A)^6}   e_4$\\

 \hline
 $\cd 4{112}$ &$\to$ &  $\cd 4{18}( \A)$ & 
 $\lambda= -\frac{1}{\Xi^2 t}$ &$\af=\frac{i(t \Xi^2 -\A)}{\A \Xi} $ & $\bt= 0 $ & $\gamma=\frac{i t \Xi}{\A}  $\\
\multicolumn{3}{l|}{$\Xi= \sqrt{1+\A+\A^2}$}& 
\multicolumn{2}{l}{$E_1^t= -i t^2 \Xi  (e_1-e_2)+(1+\A)t^2 e_3$} & \multicolumn{2}{l}{$E_2^t=t^3 (e_3-te_4)$}\\
&&& \multicolumn{2}{l}{$E_3^t= t^2 \left(e_2-\frac{it \Xi}{\A} e_3\right) $} & $E_4^t=  t^5e_4$\\

 \hline
 $\cd 4{112}$ &$\to$ &  $\cd 4{19}$ & 
 $\lambda=\frac{1}{t} $ &$\af=\frac{t}{\Xi} $ & $\bt=\frac{t}{\Xi^2}  $ & $\gamma=\frac{t}{\Xi}  $\\
 \multicolumn{3}{l|}{$\Xi= \sqrt{2-t}$} & 
 \multicolumn{1}{l}{$E_1^t= -\frac{t^2}{\Xi} e_1+\frac{t^2}{\Xi}e_2+\frac{t^2}{\Xi^2}e_3$} 
 &$E_2^t= \frac{t^3}{\Xi^2}e_3-\frac{t^4}{\Xi^4}e_4$ & \multicolumn{1}{l}{$E_3^t= -\frac{t^2}{\Xi}e_2+\frac{t^3}{\Xi^2} e_3 $} & $E_4^t= \frac{t^5}{\Xi^4}  e_4$\\

 \hline
 $\cd 4{112}$ &$\to$ &  $\cd 4{20}$ & 
 $\lambda=\frac{t-1+it\sqrt{1+t}}{t^2+t} $ &$\af=0 $ & $\bt=-\frac{it}{\sqrt{1+t}}  $ & $\gamma=0  $\\
 &&& 
 \multicolumn{1}{l}{$E_1^t=t^2(i\sqrt{1+t} e_1+e_2+e_3)$} &\multicolumn{1}{l}{$E_2^t= t^3e_2$}& \multicolumn{1}{l}{$E_3^t=  t^2e_3 $} & $E_4^t= t^5 e_4$\\

 \hline
 $\cd 4{112}$ &$\to$ &  $\cd 4{21}( \A)$ & 
\multicolumn{4}{l}{$\Xi=1-2 t+2 t^2-2 t^4+3 t^5-2 t^6+t^7+\A^2 (1-t+t^2)+\A (-2+3 t-3 t^2+2 t^3-2 t^4)$}\\
&&& \multicolumn{4}{l}{$\lambda= -\frac{t (1-2 t+3 t^2-2 t^3+t^5-t^6+t^7+\A^2 (1-t^3+t^4)+\A (-2+2 t-2 t^2+t^3+t^4-2 t^5))}{(t-1) (\A-1+t-t^2)^2}$} \\
&&&$\af=t^2$ & \multicolumn{1}{l}{$\bt=\frac{(1-t) (1-\A-t+2 t^2)}{t (1-\A-t+t^2)}  $}   &  $\gamma=1  $
  \\
 &&& \multicolumn{2}{l}{$E_1^t= t(t-1)^2\left(\frac{1-\A-t+t^2}{\Xi}  e_1+\frac{(t-\A) t^2}{\Xi}e_2-\frac{(t-1) t^2}{\Xi}e_3\right)$}& $E_4^t= \frac{(t-1)^6 t^6}{\Xi^2} e_4$ \\&&&
\multicolumn{2}{l}{ $E_2^t= \frac{(1-t)^3 t^2}{\Xi}\left(t e_3+(1-t)e_4\right)$} & 
\multicolumn{2}{l}{$E_3^t=  \frac{(1-t)^3 t^3}{\Xi}\left(te_3+e_4\right) $} \\

 \hline
 $\cd 4{112}$ &$\to$ &  $\cd 4{22}$ & 
\multicolumn{2}{l}{$\Xi=\sqrt{t^3 (-16-64 t-64 t^4+17 t^5+72 t^6+16 t^7)}$}&
\multicolumn{2}{l}{$\lambda=-\frac{8 t^3 (-8+9 t^5+4 t^6-t \Xi)}{(t^4+4 t^5-\Xi)^2} $}\\
&&&$\af=\frac{4 t^3 (-4-t+4 t^2)}{t^4+4 t^5-\Xi} $ & 
 \multicolumn{2}{l}{$\bt=\frac{32 t^3-9 t^4-40 t^5-16 t^6+(1+4t)\Xi}{4 t^4+16 t^5-4\Xi}$}   & $\gamma= 1 $\\
&&& 
$E_1^t=\frac{t^4+4 t^5-\Xi}{8} e_1+t^4e_2$ &$E_2^t=t^3 e_3-\frac{t^3(1+4t)}{4}e_4$ & $E_3^t= t^2 e_2-\frac{t^2}{2}e_3 $ & $E_4^t=  t^5e_4$\\

  \hline
 $\cd 4{112}$ &$\to$ &  $\cd 4{23}( \A)$ & 
 \multicolumn{2}{l}{$\lambda= \frac{\Xi (\Xi+t^2)}{t (\Xi+t)^2} $} &
 \multicolumn{2}{l}{$\af=\frac{(1-\A) \sqrt{t^3}}{\Xi\sqrt{ \Xi+t^2}} $}\\
 \multicolumn{3}{l|}{$\Xi=\A+\A^2$}  & 
 \multicolumn{2}{l}{$\bt=\frac{t (\Xi+t)}{\Xi+t^2}  $} & 
 \multicolumn{2}{l}{$\gamma= \frac{\sqrt{t} (\Xi+t)}{\Xi\sqrt{\Xi+t^2
 }}$}\\
 &&& 
 \multicolumn{2}{l}{$E_1^t=  \frac{\sqrt{t^5} (\Xi+t)}{\Xi^2\sqrt{ \Xi+t^2}}\left((\Xi+t) e_1-te_2+\frac{\sqrt{t^3} (\Xi+t)}{(1+\A) \sqrt{\Xi+t^2}}e_3\right)$} 
 &
 \multicolumn{2}{l}{$E_2^t= \frac{t^4 (\Xi+t)^2}{\Xi^2 (\Xi+t^2)}\left(e_3+\frac{t(\Xi+t)^2}{\Xi (\Xi+t^2)}e_4\right)$}\\
 &&& \multicolumn{2}{l}{$E_3^t=\frac{\sqrt{t^5}  (\Xi+t)}{\Xi\sqrt{ \Xi+t^2}}\left(  e_2 +\frac{\sqrt t (\Xi+t)}{ (1+\A) \sqrt{\Xi+t^2}} e_3 \right)$} & 
 \multicolumn{2}{l}{$E_4^t= \frac{t^7 (\Xi+t)^4}{\Xi^4 (\Xi+t^2)^2} e_4$}\\

 \hline
 $\cd 4{112}$ &$\to$ &  $\cd 4{24}( \A)$ & 
 \multicolumn{2}{l}{$\lambda=\frac{\A^2-t-\A t^2}{\A^2 t^3} $} &
 \multicolumn{2}{l}{$\af= \frac{\sqrt{t} (-2 \A+t+\sqrt{1-4 \A} t-2 \A^2 t^2)}{2 i\A^2 \sqrt{\A+t}} $}\\
 &&& 
 \multicolumn{2}{l}{$\bt= \frac{t^3 (2+2 \A^2+(3+\sqrt{1-4 \A}) \A t)}{2 \A (\A+t)}  $} & $\gamma=\frac{\sqrt{t^3}}{i\sqrt{\A+t}}  $\\
 &&& 
 \multicolumn{2}{l}{$E_1^t= -\frac{i\sqrt{t^5}}{\sqrt{\A+t}}\left(t e_1+\frac{1}{\A}e_2-\frac{i\sqrt{t^7}}{\sqrt{\A+t}}e_3\right)$} &
 \multicolumn{2}{l}{$E_2^t=-\frac{t^4}{\A+t} \left(e_3+ t^3e_4\right)$}\\ && & 
 \multicolumn{2}{l}{$E_3^t=  \frac{\sqrt{t^5}}{\sqrt{\A+t}}\left(-ie_2 +\frac{(1-\sqrt{1-4 \A}) \sqrt{t^3}}{2 \sqrt{\A+t}}e_3 \right)$} & $E_4^t= \frac{t^8}{(\A+t)^2}  e_4$\\

 \hline
 $\cd 4{112}$ &$\to$ &  $\cd 4{25}$ & 
 $\lambda= \frac{t-t^4-1}{t^8}$ &$\af= \frac{i(t-2)}{\sqrt{2t}} $ & $\bt=\frac{t^4(t^3-1)}{2}  $ & $\gamma=-i\sqrt{2t^7}  $\\
 &&& 
 \multicolumn{2}{l}{$E_1^t= i\sqrt{\frac{t^{7}}{2}}\left(t^4 e_1+   e_2+i\sqrt{\frac{t^9}{2}}e_3\right)$} &
 \multicolumn{2}{l}{$E_2^t= -\frac{t^8}{2}\left(e_3+\frac{t^{7}}{2}e_4\right)$}\\
 &&& 
 \multicolumn{2}{l}{$E_3^t= \sqrt{\frac{t^{11}}{2}}\left(i e_2-\sqrt{\frac{t^7}{2}}e_3\right) $} & \multicolumn{2}{l}{$E_4^t=\frac{t^{16}}{4}  e_4$}\\
 
  \hline
 $\cd 4{44}$ &$\to$ &  $\cd 4{26}(\A)$ & 
  & $\af=0$  & $\bt= \A t^{-3}$ & $\gamma=t^{-3}$\\
 &&& $E_1^t=  t^{-1}(e_1 - e_3)$ &$E_2^t= t^{-2}e_3$ & $E_3^t= e_2$ & $E_4^t= t^{-4}e_4$\\


  \hline
 $\cd 4{112}$ &$\to$ &  $\cd 4{27}$ & 
 $\lambda=-2+t^{-4} $ &$\af=-\frac{\sqrt{t}}{\Xi} $ & $\bt=\frac{2 t^4-2 t^5-\sqrt{t^5}}{\Xi^2}  $ & $\gamma= 0 $\\
 \multicolumn{3}{l|}{$\Xi=\sqrt{1-t}$}&
 $E_1^t= \frac{\sqrt{t^{11}}}{\Xi} \left(e_1+e_2\right)+\frac{t^3}{\Xi^2}e_3$ &
 $E_2^t= \frac{t^7}{\Xi^2}\left(e_3+t^{4}e_4\right)$
 &$E_3^t=\frac{t^4}{\Xi}\left(e_2+\frac{t^2}{\Xi} e_3\right) $ & $E_4^t= \frac{t^{13}}{\Xi^4} e_4$\\

  \hline
 $\cd 4{112}$ &$\to$ &  $\cd 4{28}(\A)$ & 
 $\lambda=-2+t^{-2} $ &$\af= \frac{-\sqrt{t}+t-\A t}{\Xi}$ & $\bt=-\frac{\sqrt{t^3}}{\Xi^2} +2t^2 $ & $\gamma=0  $\\
 \multicolumn{3}{l|}{$\Xi=\sqrt{1-\A}$}& $E_1^t=\frac{\sqrt{t^5}}{\Xi}\left(e_1+e_2\right)+\frac{t^3}{\Xi^2}e_3$ &
 $E_2^t= \frac{t^3}{\Xi^2}\left(e_3+t^2e_4\right)$ & $E_3^t= \frac{t^2}{\Xi} \left(e_2+\frac{t}{\Xi}e_3\right)$ & $E_4^t= \frac{t^6}{\Xi^4} e_4$\\

 \hline
 $\cd 4{112}$ &$\to$ &  $\cd 4{29}$ & 
 $\lambda=-t^{-5} $ &$\af=  2\Xi$ & $\bt= - 4 \Xi^2 $ & $\gamma=- 2 \Xi  $\\
 \multicolumn{3}{l|}{$\Xi=\frac{t^5}{\sqrt{1+4t^2}}$}& 
 $E_1^t=  2 t^8 \Xi\left(-e_1+e_2+4\Xi e_3\right)$ &
 $E_2^t= -4 t^6 \Xi^2 \left(e_3-t^{10} e_4\right)$
 & $E_3^t= 2 t^4\Xi \left(e_2+\Xi e_3\right) $ & $E_4^t= 16t^{10}\Xi^4 e_4$\\

 \hline
 $\cd 4{112}$ &$\to$ &  $\cd 4{30} $ & 
 $\lambda=-2+t^{-5} $ &$\af=-\frac{i\sqrt{t}}{\Xi}  $ & $\bt=\frac{t^3-2 t^5}{\Xi^2}  $ & $\gamma=0  $\\
 \multicolumn{3}{l|}{$\Xi=\sqrt{2 t^5-t^3-1}$}& 
$E_1^t= \frac{i \sqrt{t^{13}}}{\Xi} \left( e_1+e_2\right)-\frac{t^7}{\Xi^2}e_3$ &
$E_2^t=-\frac{t^8}{\Xi^2}\left( e_2-\frac{t^{5}}{\Xi^2}e_3\right)$ & 
$E_3^t=\frac{\sqrt{t^9}}{\Xi}\left(i e_2 -\frac{t^5}{\Xi} e_3\right) $ & $E_4^t= \frac{t^{15}}{\Xi^4}e_4$\\

 \hline
 $\cd 4{112}$ &$\to$ &  $\cd 4{31}$ & 
 $\lambda=-2+t^{-2} $ &$\af=\frac{i t (1-\sqrt{t^3})}{\Xi} $ & $\bt= -\frac{2 t^2+\sqrt{t^5}-\sqrt{t^7}}{\Xi^2} $ & $\gamma= 0 $\\
 \multicolumn{3}{l|}{$\Xi=\sqrt{2 t^2-1}$}& 
 $E_1^t= \frac{i \sqrt{t^7}}{\Xi}  (e_1+ e_2)$ &$E_2^t= -\frac{t^5}{\Xi^2}\left(e_3-\frac{t^2}{\Xi^2}e_4\right)$ & $E_3^t= \frac{t^3}{\Xi}  \left(ie_2-\frac{t}{\Xi}e_3\right) $ & $E_4^t= \frac{t^9}{\Xi^4} e_4$\\

 \hline
 $\cd 4{112}$ &$\to$ &  $\cd 4{32}( \A)$ & 
 $\lambda= -2 +t^{-5}$ &$\af=-\frac{2 i \sqrt{t^5}}{\Xi} $ & $\bt=-\frac{(1-2 \A) t^5}{\A+t^5-2 \A t^5}  $ & $\gamma=\frac{i \sqrt{t^5}}{\Xi}  $\\
\multicolumn{3}{l|}{$\Xi=\sqrt{(1+\A) ( 2 \A t^5-t^5-\A) }$}& 
\multicolumn{1}{l}{$E_1^t= \frac{i \sqrt{t^{13}}}{\Xi}\left( e_1+  e_2\right)-\frac{t^9}{\Xi^2}e_3$} &
\multicolumn{2}{l}{$E_2^t=-\frac{t^8}{\Xi^2} \left(e_3-\frac{\A(1+\A)t^{5}}{\Xi^2}e_4\right)$}\\
&&&
\multicolumn{1}{l}{$E_3^t= \frac{i \sqrt{t^9}}{\Xi} \left(i e_2-\frac{\A \sqrt{t^5}}{\Xi}e_3\right) $} & \multicolumn{2}{l}{$E_4^t= \frac{t^{15}}{\Xi^4} e_4$}\\

  \hline
 $\cd 4{112}$ &$\to$ &  $\cd 4{33}$ & 
 $\lambda=-2+t^{-2} $ &$\af=-i(t^3-\sqrt{t}) $ & $\bt= 2 t^2  $ & $\gamma= it^3  $\\
 &&& 
 $E_1^t= -i\sqrt{t^5}  (e_1+e_2)-t^3e_3$ & $E_2^t= -t^3(e_3+t^2 e_4)$
 &
 $E_3^t=-t^2(i e_2 +t^3e_3)$ & $E_4^t=  t^6e_4$\\

 \hline
 $\cd 4{112}$ &$\to$ &  $\cd 4{34}$ & 
 $\lambda=-2+t^{-2} $ &$\af=\frac{-\sqrt{t^3}+t^2}{\Xi} $ & $\bt= \frac{2 t^2}{\Xi} $ & $\gamma=-\frac{t^2}{\Xi}  $\\
\multicolumn{3}{l|}{$\Xi=\sqrt{2 t^2-1}$}& 
$E_1^t= \frac{ \sqrt{t^5}}{\Xi} (e_1+e_2)+\frac{t^4}{\Xi^2}e_3$ &
$E_2^t=\frac{t^3}{\Xi^2} \left(e_3-\frac{t^2}{\Xi^2}e_4\right)$
&
$E_3^t=\frac{t^2}{\Xi}  \left(e_2+\frac{t^2}{\Xi}e_3\right) $ & $E_4^t= \frac{t^6}{(1+t)\Xi^4} e_4$\\

 \hline
 $\cd 4{112}$ &$\to$ &  $\cd 4{35}$ & 
 $\lambda= t^{-1}$ &$\af=0 $ & $\bt=0  $ & $\gamma=-i \sqrt{t^3}  $\\
 &&& $E_1^t= i \sqrt{t^3}  e_1$ &$E_2^t= -t^2(e_3+t e_4)$ & $E_3^t= \sqrt{t^3}(i e_2 -\sqrt{t^3} e_3)$ & $E_4^t=t^4  e_4$\\

  \hline
 $\cd 4{112}$ &$\to$ &  $\cd 4{36}( \A)$ & 
 $\lambda=-\frac{1}{\A} $ &$\af=-\frac{2\sqrt{t^3}}{\A \sqrt{t^3}+\Xi} $ & $\bt= -\frac{\A((2+\A)\sqrt{t^3}+\Xi)}{\A\sqrt{t^3}+\Xi} $ & $\gamma=1  $\\
 \multicolumn{3}{l|}{$\Xi=\sqrt{\A (-4+(4+\A) t^3)}$}& 
$E_1^t=  \frac{A t^2 +\sqrt{t}\Xi}{2} e_1+t^2e_2$ &$E_2^t= t(e_3- \A e_4)$ & $E_3^t= t e_2+\frac{(\sqrt{1-4\A}-1)\A t}{2}e_4 $ & $E_4^t=  t^2e_4$\\


 \hline
 $\cd 4{112}$ &$\to$ &  $\cd 4{37}$ & 
 $\lambda=t^{-1} $ &$\af= \frac{it}{\sqrt{2}}$ & $\bt=0  $ & $\gamma= \frac{it}{\sqrt{2}} $\\
 &&& 
 \multicolumn{1}{l}{$E_1^t= -\frac{it^2}{\sqrt{2}}  \left(e_1-e_2-\frac{i}{\sqrt{2}}e_3\right)$} &
 \multicolumn{1}{l}{$E_2^t= -\frac{t^3}{2}\left(e_3+\frac{t}{2}e_4\right)$}& 
 \multicolumn{1}{l}{$E_3^t=  \frac{t^2}{\sqrt{2}}\left(ie_2+\frac{t}{\sqrt 2}e_3\right) $} & $E_4^t=\frac{t^5}{4}  e_4$\\


 \hline
 $\cd 4{112}$ &$\to$ &  $\cd 4{38}$ & 
 $\lambda= 1$ &$\af=0 $ & $\bt=1  $ & $\gamma=0  $\\
 &&& $E_1^t=\sqrt{t^3}   e_1$ &$E_2^t= t^3(e_3+e_4)$ & $E_3^t=  t^2(e_2+e_3) $ & $E_4^t= t^5 e_4$\\


  \hline
 $\cd 4{44}$ &$\to$ &  $\cd 4{39}$ & 
  \multicolumn{2}{r}{$\af=\frac 1{2t^2}(2t^3 + 3t^2 + 24t - 36)$} & $\bt= \frac 9{t^2}\sqrt{\Xi} $ & $\gm=-\frac {\sqrt{\Xi}}{2t}(t + 6)
  $\\
  \multicolumn{3}{l|}{$\Xi=t-1$} & \multicolumn{1}{l}{$E_1^t= \Xi(e_1 + \frac 1{\sqrt{\Xi}}e_2 + e_3)$} &\multicolumn{2}{l}{$E_2^t=i\Xi(e_1-\sqrt{\Xi}e_2-(\frac t2 + 2)e_3-\frac {3\Xi}{2t}(t + 3)e_4)$} \\
 &&&\multicolumn{1}{l}{$E_3^t=\Xi(te_3+\frac {3\Xi}t(t + 3)e_4)$} & \multicolumn{2}{l}{$E_4^t=t\Xi^2e_4$}\\
 
 \hline
 $\cd 4{44}$ &$\to$ &  $\cd 4{40}$ & 
&  \multicolumn{1}{l}{$\af=t - \frac 34$} & $\bt= 0 $ & $\gm=\sqrt{\Xi}$\\
   \multicolumn{3}{l|}{$\Xi=t-1$}& \multicolumn{1}{l}{$E_1^t= \Xi(e_1 + \frac 1{\sqrt{\Xi}}e_2 -\frac 12e_3)$} &\multicolumn{2}{l}{$E_2^t=i\Xi(e_1-\sqrt{\Xi}e_2+(t - \frac 12)e_3)$} \\
 &&&\multicolumn{1}{l}{$E_3^t=t\Xi e_3$} & \multicolumn{2}{l}{$E_4^t=t\Xi^2e_4$}\\
 
\hline
 $\cd 4{44}$ &$\to$ &  $\cd 4{41}$ & 
  \multicolumn{2}{r}{$\af=\frac 1{2t^2}(2t^3 + 13t^2 + 8t - 36)$} & $\bt= \frac 1{t^2}(4t + 9)\sqrt{\Xi} $ & $\gm=-\frac {3\sqrt{\Xi}}{2t}(t + 2)$\\
   \multicolumn{3}{l|}{$\Xi=t-1$} & \multicolumn{1}{l}{$E_1^t= \Xi(e_1 + \frac 1{\sqrt{\Xi}}e_2 + e_3)$} &\multicolumn{2}{l}{$E_2^t=i\Xi(e_1-\sqrt{\Xi}e_2-(\frac 32t + 2)e_3-\frac {3\Xi}{2t}(7t + 9)e_4)$} \\
 &&&\multicolumn{1}{l}{$E_3^t=\Xi(te_3+\frac \Xi t(7t + 9)e_4)$} & \multicolumn{2}{l}{$E_4^t=t\Xi^2e_4$}\\
 
 \hline
 $\cd 4{44}$ &$\to$ &  $\cd 4{42}$ & 
 & \multicolumn{1}{l}{$\af=\frac 1{4t}(4t^2 - 19t + 16)$} & $\bt= -\frac 2t\sqrt{\Xi} $ & $\gm=3\sqrt{\Xi}$\\
    \multicolumn{3}{l|}{$\Xi=t-1$} & \multicolumn{1}{l}{$E_1^t= \Xi(e_1 + \frac 1{\sqrt{\Xi}}e_2 -\frac 12e_3)$} &\multicolumn{2}{l}{$E_2^t=i\Xi(e_1-\sqrt{\Xi}e_2+(3t - \frac 12)e_3-6\Xi e_4)$} \\
 &&&\multicolumn{1}{l}{$E_3^t=\Xi(te_3-2\Xi e_4)$} & \multicolumn{2}{l}{$E_4^t=t\Xi^2e_4$}\\
 
 \hline
 $\cd 4{44}$ &$\to$ &  $\cd 4{43}(\A)$ & 
  &$\af=\frac{1}{4}+t-2 \A t^2 $ & $\bt= \A t \sqrt{1-t^2} $ & $\gamma=0  $\\
 \multicolumn{3}{l|}{$\Xi=t\sqrt{1-t^2}$}  & $E_1^t= \frac{\Xi^2}t\left(e_1+\frac {t^2}{\Xi}e_2-\frac{1}{2}e_3\right)$ & $E_2^t=\Xi e_2$
 & $E_3^t=\Xi^2 e_3 $ & $E_4^t= \frac{\Xi^4}t e_4$\\
 
 \hline
 $\cd 4{112}$ &$\to$ &  $\cd 4{44}(\A,\B,\C)$ & 
  $\lambda=\frac{\rho^2-t}{t^2}$ &$\af=\frac{i\Xi(t-\rho^2)}{\sqrt{\rho(\rho^2 + 1)}} $ & $\bt=\Xi^2(\rho\A + (\rho^2-1)\B)t  $ & $\gamma=\frac{i\Xi t}{\sqrt{\rho(\rho^2 + 1)}}$\\
    \multicolumn{3}{l|}{$\Xi=\frac 1{\sqrt{\rho^2\B(t + 2) + \rho\A(t + 1) - \B t}}$} & \multicolumn{2}{l}{$E_1^t= \frac{i\Xi t}{\sqrt{\rho(\rho^2 + 1)}}(t e_1 - (t + 1)e_2)-\frac{\Xi^2t^2}{\rho} e_3$} &\multicolumn{2}{l}{$E_2^t=-\frac{i\Xi t}{\rho\sqrt{\rho(\rho^2 + 1)}}(t e_1-(t-\rho^2)e_2)$} \\
 \multicolumn{3}{l|}{$\rho=-\frac {\A + \sqrt{\A^2 + 4\B(\B + \C)}}{2\B}$}&\multicolumn{2}{l}{$E_3^t=-\frac{\Xi^2 t^2}{\rho}\left(
e_3+\frac{\Xi^2 (\A+\rho\B)t^2}{\rho} e_4\right)$} & \multicolumn{2}{l}{$E_4^t=\frac{\Xi^4 t^4}{\rho^2} e_4$}\\
 
 \hline
 $\cd 4{44}$ &$\to$ &  $\cd 4{45}$ & 
  \multicolumn{2}{r}{$\af=-\Xi^2(t^3 - 2it^2 + 10it - 8i)$} & $\bt= i\Xi^{\frac 32}t^{-\frac 12}(3t - 4) $ & $\gm=0$\\
    \multicolumn{3}{l|}{$\Xi=t^{-1}(t-1)^{-1}$} & \multicolumn{2}{l}{$E_1^t= -t^{-1} e_1-\sqrt{\Xi t^{-1}}e_2+t\Xi e_3$} &\multicolumn{2}{l}{$E_2^t=i(-t^{-1} e_1+t^{-\frac 32}\Xi^{-\frac 12} e_2+t\Xi e_3)$} \\
 &&&\multicolumn{2}{l}{$E_3^t=\Xi (e_3-4i\Xi t^{-1}e_4)$} & \multicolumn{2}{l}{$E_4^t=\Xi^2e_4$}\\
 
  \hline
 $\cd 4{44}$ &$\to$ &  $\cd 4{46}(\A)$ & 
  & \multicolumn{1}{l}{$\af=-2i\A\Xi$} & $\bt= -i\A\sqrt{\Xi t^{-1}}$ & $\gm=0$\\
    \multicolumn{3}{l|}{$\Xi=t^{-1}(t-1)$} & \multicolumn{1}{l}{$E_1^t=\Xi e_1 + \sqrt{\Xi t^{-1}}e_2 $} &\multicolumn{1}{l}{$E_2^t=i\Xi(	
e_1-\sqrt {\Xi t} e_2)$}  &\multicolumn{1}{l}{$E_3^t=\Xi e_3$} & \multicolumn{1}{l}{$E_4^t=\Xi^2 e_4$}\\
 
 \hline
 $\cd 4{44}$ &$\to$ &  $\cd 4{47}( \A, \B)$ & 
  & $\af=-\A-\A^2+\B^2+t $  & $\bt= -\A\B $ & $\gamma=-\B   $\\
 &&& $E_1^t=  t (e_1+ \A e_3)$ &$E_2^t= t(e_2+\B e_3)$ & $E_3^t= t^2 (e_3 +\B^2 e_4)$ & $E_4^t= t^3 e_4$\\
 
\hline
 $\cd 4{44}$ &$\to$ &  $\cd 4{48}(\A)$ & 
 & \multicolumn{1}{l}{$\af=\A\Xi$} & $\bt= \A\sqrt{\Xi t^{-1}}$ & $\gm=-\A\sqrt{\Xi t^{-1}}$\\
    \multicolumn{3}{l|}{$\Xi=t^{-1}(t-1)$} & \multicolumn{1}{l}{$E_1^t=\Xi e_1 + \sqrt{\Xi t^{-1}}e_2 $} &\multicolumn{1}{l}{$E_2^t=i\Xi(	
e_1-\sqrt {\Xi t} e_2)$} &\multicolumn{1}{l}{$E_3^t=\Xi e_3$} & \multicolumn{1}{l}{$E_4^t=\Xi^2 e_4$}\\
 
 \hline
 $\cd 4{44}$ &$\to$ &  $\cd 4{49}(\A)$ & 
  \multicolumn{2}{r}{$\af=-\Xi^2(t^3-\A t^2 + 5\A t - 4\A)$} & $\bt= -\A(t - 2)\Xi^{\frac 32}t^{-\frac 12}$ & $\gm=-\A\Xi^{\frac 32}t^{\frac 12}$\\
    \multicolumn{3}{l|}{$\Xi=t^{-1}(t-1)^{-1}$} & \multicolumn{2}{l}{$E_1^t=-t^{-1}e_1 + \sqrt{\Xi t^{-1}}e_2+t\Xi e_3 $} &\multicolumn{2}{l}{$E_2^t=-i(t^{-1}e_1 +\Xi^{-\frac 12}t^{-\frac 32}e_2-t\Xi e_3)$} \\
 &&&\multicolumn{2}{l}{$E_3^t=\Xi e_3-2\A\Xi^2 t^{-1}e_4$} & \multicolumn{2}{l}{$E_4^t=\Xi^2 e_4$}\\
 
 \hline
 $\cd 4{44}$ &$\to$ &  $\cd 4{50}( \A)$ & 
   &$\af= \A t^{-2}$ & $\bt=0  $ & $\gamma=t^{-2}  $\\
 &&& $E_1^t=  t^{-1}  e_1$ &$E_2^t=t^{-1} e_2$ & $E_3^t= t^{-2} e_3 $ & $E_4^t=t^{-4}  e_4$\\
 
  \hline
 $\cd 4{44}$ &$\to$ &  $\cd 4{51}(\A)$ & 
  &\multicolumn{1}{l}{$\af=\A\Xi t^{-1}(t - 2)$} & $\bt= \A\Xi^{\frac 32}t^{-\frac 12}$ & $\gm=0$\\
    \multicolumn{3}{l|}{$\Xi=t^{-1}(t-1)$} & \multicolumn{1}{l}{$E_1^t=\Xi e_1 + \sqrt{\Xi t^{-1}}e_2 $} &\multicolumn{1}{l}{$E_2^t=i\Xi(	
e_1-\sqrt {\Xi t} e_2)$} &\multicolumn{1}{l}{$E_3^t=\Xi (e_3-\A \Xi t^{-1}e_4)$} & \multicolumn{1}{l}{$E_4^t=\Xi^2 e_4$}\\
 
 \hline
 $\cd 4{44}$ &$\to$ &  $\cd 4{52}$ & 
  &$\af= t^{-2}$ & $\bt=0  $ & $\gamma=0  $\\
 &&& $E_1^t=t^{-1}   e_1$ &$E_2^t=t^{-1} e_2$ & $E_3^t=t^{-2}  e_3 $ & $E_4^t=t^{-4}  e_4$\\
 
 \hline
 $\cd 4{44}$ &$\to$ &  $\cd 4{53}$ & 
  &$\af=-\frac{2+t}{t^2} $ & $\bt=-\frac{1}{2}-\frac{i}{t^2}  $ & $\gamma= 1-\frac{i}{t} $\\
 &&& $E_1^t=  t e_1+e_3$ &$E_2^t=t e_2+ie_3$ & $E_3^t=  t^2e_3-e_4 $ & $E_4^t=t^2  e_4$\\
 
 \hline
 $\cd 4{44}$ &$\to$ &  $\cd 4{54}(\A)$ & 
  \multicolumn{2}{r}{$\af=\frac{-\A(\A + 1)t^2 + 2(2\A + 1)^2(t - 1) + 4i\Xi}{t^2}$} &  \multicolumn{2}{l}{$\bt= \frac{(2\A + 1)(\A(t - 2) - 1)\Xi + it^2(t - 2)}{t^3}$ $\gm=\frac{(2\A + 1)\Xi}{t^2}$} \\
    \multicolumn{3}{l|}{$\Xi=t\sqrt{t-1}$} & \multicolumn{2}{l}{$E_1^t=-\frac 1t(\Xi e_1 - te_2 + \Xi\A e_3)$} &\multicolumn{2}{l}{$E_2^t=-\frac it(\Xi e_1+t(t-1)e_2-\Xi(\A + 1)e_3)$} \\
 &&&\multicolumn{2}{l}{$E_3^t=te_3+\frac\Xi {t^3}((2\A+1)^2\Xi + 2it^2)e_4$} & \multicolumn{2}{l}{$E_4^t=-\Xi e_4$}\\
 
 \hline
 $\cd 4{44}$ &$\to$ &  $\cd 4{55}( \A)$ & 
  &$\af= -\A (1+\A) $ & $\bt=t  $ & $\gamma= 0 $\\
 &&& $E_1^t=  t e_1+\A t e_3$ &$E_2^t=t e_2$ & $E_3^t= t^2 e_3 $ & $E_4^t=t^3  e_4$\\
 
 \hline
 $\cd 4{44}$ &$\to$ &  $\cd 4{56}$ & 
 & \multicolumn{1}{l}{$\af=0$} & $\bt= i\sqrt{\Xi t^{-1}}$ & $\gm=-2i\sqrt{\Xi t^{-1}}$\\
    \multicolumn{3}{l|}{$\Xi=t^{-1}(t-1)$} & \multicolumn{1}{l}{$E_1^t=\Xi e_1 + \sqrt{\Xi t^{-1}}e_2 $} &\multicolumn{1}{l}{$E_2^t=i\Xi(	
e_1-\sqrt {\Xi t} e_2)$}  &\multicolumn{1}{l}{$E_3^t=\Xi e_3$} & \multicolumn{1}{l}{$E_4^t=\Xi^2 e_4$}\\
 
 \hline
 $\cd 4{44}$ &$\to$ &  $\cd 4{57}(\A,\B)$ & 
&  \multicolumn{1}{l}{$\af=-2i(2\A + \B)\Xi t^{-1}$} & $\bt= i(\A (t - 2) - \B)\sqrt{\Xi t^{-3}}$ & $\gm=i\A\sqrt{\Xi t^{-1}}$\\
    \multicolumn{3}{l|}{$\Xi=t^{-1}(t-1)$} & \multicolumn{2}{l}{$E_1^t=\Xi e_1 + \sqrt{\Xi t^{-1}}e_2 $} &\multicolumn{2}{l}{$E_2^t=i\Xi(	
e_1-\sqrt {\Xi t} e_2)$} \\
 &&&\multicolumn{2}{l}{$E_3^t=\Xi (e_3-i(2\A + \B)\Xi t^{-1}e_4)$} & \multicolumn{2}{l}{$E_4^t=\Xi^2 e_4$}\\
 
  \hline
 $\cd 4{44}$ &$\to$ &  $\cd 4{58}$ & 
  &$\af=-\frac{2+t}{t^2} $ & $\bt=-\frac{i}{t^2}  $ & $\gamma= \frac{t-i}{t} $\\
 &&& $E_1^t=  t e_1+e_3$ &$E_2^t= te_2+ie_3$ & $E_3^t=t^2  e_3-e_4 $ & $E_4^t=t^2  e_4$\\

 \hline
 $\cd 4{44}$ &$\to$ &  $\cd 4{59}(\A, \B)$ & 
  &$\af= -\A-\A^2+\B^2 $ & $\bt= -\A\B $ & $\gamma=-\B  $\\
 &&& $E_1^t=  t (e_1+\A e_3)$ &$E_2^t= t(e_2+\B e_3)$ & $E_3^t=t^2  (e_3+\B^2 e_4) $ & $E_4^t=t^3  e_4$\\
 
  \hline
 $\cd 4{44}$ &$\to$ &  $\cd 4{60}$ & 
  \multicolumn{2}{r}{$\af=t^{-2}((1 - i)t^2 + 2(i - 4)t + 8)$} & $\bt= 2\Xi t^{-2}(t - 2)$ & $\gm=-2i\Xi t^{-1}$\\
    \multicolumn{3}{l|}{$\Xi=\sqrt{t-1}$} & 
    \multicolumn{1}{l}{$E_1^t=\Xi e_1 + e_2 + i\Xi e_3$} &\multicolumn{1}{l}{$E_2^t=	
\Xi(i e_1-i\Xi e_2+ e_3)$}
&\multicolumn{1}{l}{$E_3^t=te_3-4t^{-1}\Xi^2 e_4$} & \multicolumn{1}{l}{$E_4^t=t\Xi e_4$}\\
 
 \hline
 $\cd 4{44}$ &$\to$ &  $\cd 4{61}(\A)$ & 
&  \multicolumn{1}{l}{$\af=\A(\A + i)$} & $\bt= 0$ & $\gm=0$\\
    \multicolumn{3}{l|}{$\Xi=\sqrt{t-1}$} & \multicolumn{1}{l}{$E_1^t=\Xi e_1 + e_2 - i\A\Xi e_3$} &\multicolumn{1}{l}{$E_2^t=	
\Xi(ie_1-i\Xi e_2+\A e_3)$} 
&\multicolumn{1}{l}{$E_3^t=te_3$} & \multicolumn{1}{l}{$E_4^t=t\Xi e_4$}\\
 
 \hline
 $\cd 4{44}$ &$\to$ &  $\cd 4{62}$ & 
 & \multicolumn{1}{l}{$\af=-2t^{-2}(t^2 - 9t + 9)$} & $\bt=-3\Xi t^{-2}(t - 3)$ & $\gm=-3\Xi t^{-1}$\\
    \multicolumn{3}{l|}{$\Xi=\sqrt{t-1}$} & 
    \multicolumn{1}{l}{$E_1^t=\Xi e_1 + e_2 + \Xi e_3$} &\multicolumn{1}{l}{$E_2^t=i\Xi(	
e_1-\Xi e_2-2e_3)$}  &\multicolumn{1}{l}{$E_3^t=te_3+ 9\Xi^2 t^{-1}e_4$} & \multicolumn{1}{l}{$E_4^t=t\Xi e_4$}\\

\hline
 $\cd 4{44}$ &$\to$ &  $\cd 4{63}$ & 
&  \multicolumn{1}{l}{$\af=\frac 14$} & $\bt=0$ & $\gm=0$\\
    \multicolumn{3}{l|}{$\Xi=\sqrt{t-1}$} & 
    \multicolumn{1}{l}{$E_1^t=\Xi e_1 + e_2 -\frac 12 \Xi e_3$} &\multicolumn{1}{l}{$E_2^t=	i\Xi(e_1-\Xi e_2-\frac 12 e_3)$} &\multicolumn{1}{l}{$E_3^t=te_3$} 
    &\multicolumn{1}{l}{$E_4^t=t\Xi e_4$}\\
 
  \hline
 $\cd 4{44}$ &$\to$ &  $\cd 4{64}(\A)$ & 
  &$\af=-\A^2-\A $ & $\bt=0  $ & $\gamma=0  $\\
 &&& $E_1^t=t   (e_1+\A e_3)$ &$E_2^t=t e_2$ & $E_3^t=  t^2e_3 $ & $E_4^t=t^3  e_4$\\
 
 \hline
 $\cd 4{44}$ &$\to$ &  $\cd 4{65}(\A)$ & 
  \multicolumn{2}{r}{$\af=-\A((\A + 1)t - 2\A)t^{-2}\Xi^2$} & $\bt=-\A^2\Xi^3 t^{-2}$ & $\gm=-\A\Xi t^{-1}$\\
    \multicolumn{3}{l|}{$\Xi=\sqrt{t-1}$} & 
    \multicolumn{1}{l}{$E_1^t=\Xi e_1 + e_2 +\A \Xi e_3$} &\multicolumn{1}{l}{$E_2^t=	
i\Xi(e_1-\Xi e_2)$} 
&\multicolumn{1}{l}{$E_3^t=te_3+\A^2\Xi^2 t^{-1}e_4$} & \multicolumn{1}{l}{$E_4^t=t\Xi e_4$}\\
 
  \hline
 $\cd 4{44}$ &$\to$ &  $\cd 4{66}$ & 
  &$\af=t^{-2} $ & $\bt=t^{-1}  $ & $\gamma=0  $\\
 &&& $E_1^t=   e_1-e_3$ &$E_2^t= e_2+t^{-1}e_3$ & $E_3^t=  e_3 +t^{-2}e_4$ & $E_4^t=  t^{-1}e_4$\\
 
  \hline
 $\cd 4{44}$ &$\to$ &  $\cd 4{67}$ & 
 &$\af=0 $ & $\bt=0  $ & $\gamma=0  $\\
 &&& $E_1^t= t^{-1}  e_1$ &$E_2^t=t^{-1} e_2$ & $E_3^t=t^{-2}  e_3 $ & $E_4^t=t^{-4}  e_4$\\

 \hline
 $\cd 4{44}$ &$\to$ &  $\cd 4{68}$ & 
   &$\af=\frac{2-it+t^2}{t^2} $ & $\bt= -\frac{i}{t^2} $ & $\gamma=-\frac{1}{t}  $\\
 &&& $E_1^t=   e_1+i e_3$ &$E_2^t= te_2+e_3$ & $E_3^t=  t^2e_3+e_4 $ & $E_4^t= t^2 e_4$\\

 \hline
 $\cd 4{44}$ &$\to$ &  $\cd 4{69}$ & 
  &$\af=\frac{2-it}{t^2} $ & $\bt= -\frac{i}{t^2} $ & $\gamma=-\frac{1}{t}  $\\
 &&& $E_1^t=  t e_1+ie_3$ &$E_2^t= t e_2+e_3$ & $E_3^t= t^2 e_3+e_4 $ & $E_4^t=t^2  e_4$\\
 
  \hline
 $\cd 4{44}$ &$\to$ &  $\cd 4{70}$ & 
  &$\af=-\frac{1+t}{t^2} $ & $\bt= 0 $ & $\gamma=0  $\\
 &&& $E_1^t=te_1+e_3   $ &$E_2^t= te_2$ & $E_3^t= t^2 e_3 $ & $E_4^t=t^2  e_4$\\

 \hline
 $\cd 4{12}$ &$\to$ &  $\cd 4{71}$ & 
 $\lambda=-1+t $ &$\af= -t^2$ &  \\
 &&& $E_1^t=  t (e_1+i e_2)$ &$E_2^t= t(e_2+i e_3+3 te_4)$ & $E_3^t= t^2 (e_3 -3 i t e_4)$ & $E_4^t= (t-3) t^3 e_4$\\
 
  \hline
 $\cd 4{12}$ &$\to$ &  $\cd 4{72}$ & 
 $\lambda=-1 $ &$\af= t^2$ &  \\
 &&& $E_1^t=  t e_1 $ &$E_2^t= -3(e_2-3 e_3-36e_4)$ & $E_3^t= -3t (e_3 +9 e_4)$ & $E_4^t= 9 t^2 e_4$\\
 
   \hline
 $\cd 4{12}$ &$\to$ &  $\cd 4{73}$ & 
 $\lambda=-1 $ &$\af= \frac{t^2}{4}$ &  \\
 &&& $E_1^t=  t e_1-6e_2 $ &$E_2^t= t(e_2+9e_4)$ & $E_3^t= t^2 \left(e_3 +\frac{3t}{2}  e_4\right)$ & $E_4^t= -3 t^3 e_4$\\
 
   \hline
 $\cd 4{112}$ &$\to$ &  $\cd 4{74}(\A)$ & 
 $\lambda=\frac{1}{4} $ &$\af= -\frac{1+2 \A}{\Xi}$ &  $\bt=0$  &$\gamma = 0$  \\
 \multicolumn{3}{l|}{$\Xi=\sqrt{\A (1+\A)}$}& 
 $E_1^t=  t e_1-6e_2 $ &$E_2^t= t(e_2+9 e_4)$ & $E_3^t= t^2 \left(e_3 +\frac{3t}{2}  e_4\right)$ & $E_4^t= -3 t^3 e_4$\\
 
   \hline
 $\cd 4{112}$ &$\to$ &  $\cd 4{75}(\A)$ & 
 $\lambda=\frac{1}{4} $ &$\af= -\frac{i}{\Xi}$ &  $\bt=4$  &$\gamma =\frac{4i\A}{\Xi}$  \\
 \multicolumn{3}{l|}{$\Xi=\sqrt{\A (1+\A)}$}& 
 $E_1^t=  8 i \Xi t\left(2 e_1- e_2+\frac{2i\A}{\Xi} e_3\right) $ &
  $E_2^t= 16 t^2 (i \Xi  e_2+4\A e_3)$
  & $E_3^t= 128 \Xi^2 t^3(e_3 +8e_4)$ & $E_4^t=2048 \Xi^2 t^4 e_4$\\

   \hline
 $\cd 4{12}$ &$\to$ &  $\cd 4{76}$ & 
 $\lambda=-1 $ &$\af= -3t^2$ &  \\
 &&& $E_1^t=  \frac{2}{\Xi} e_1-\frac{1}{\Xi}e_2 +\frac{2}{1+A} e_3$ &$E_2^t=\frac{2 t}{\Xi} e_2$ & $E_3^t= -\frac{2 t}{\Xi^2} e_3 $ & $E_4^t= \frac{4t}{\Xi^4}  e_4$\\

  \hline
 $\cd 4{78}$ &$\to$ &  $\cd 4{77}$ 
       & $E_1^t=  t^{-1} e_1 $ &$E_2^t=t^{-1}  e_2$ & $E_3^t=  t^{-2}e_3 $ & $E_4^t= t^{-3}  e_4$\\
 
\hline
 $\cd 4{112}$ &$\to$ &  $\cd 4{78}$ & 
 $\lambda=\frac 14 $ &$\af= \frac i{\sqrt {2t}}(1-2t)$ &  $\bt=2$  &$\gamma =i\sqrt{\frac 2t}$  \\
&&& 
 \multicolumn{1}{l}{$E_1^t= -2i\sqrt{2t}(2e_1-e_2)$} & \multicolumn{2}{l}{$E_2^t= -2i\sqrt{2t}(2 e_1-(2t + 1)e_2) + 8te_3$}\\ 
&&&  \multicolumn{1}{l}{$E_3^t= -16 t^2 (e_3 +4 e_4)$} & \multicolumn{2}{l}{$E_4^t=128t^3 e_4$}\\
 
    \hline
 $\cd 4{112}$ &$\to$ &  $\cd 4{79}(\A)$ & 
 $\lambda=\frac{1+2 t^2}{4} $ &$\af= -\frac{1+t}{\Xi}$ &  $\bt=2$  &$\gamma =\frac{2(t-1)}{\Xi}$  \\
 \multicolumn{3}{l|}{$\Xi=\sqrt{2t(\A -t)}$}& 
 $E_1^t= \frac{2}{\Xi}(2e_1-e_2) $ &
  $E_2^t= \frac{4t}{\Xi} \left(e_2+\frac{2t}{\Xi}e_3\right)$
  & 
  $E_3^t= \frac{8t}{\Xi^2}(e_3 +4e_4)$ & $E_4^t=\frac{64t^2}{\Xi^4}e_4$\\
 
    \hline
 $\cd 4{112}$ &$\to$ &  $\cd 4{80}$ & 
 $\lambda=\frac{1}{4} $ &$\af= -\sqrt{2t}$ &  $\bt=2$  &$\gamma =0$  \\
 &&& 
 $E_1^t= \sqrt{\frac{2}{t}}(2e_1-e_2) $ &
  $E_2^t= 2\sqrt{2t}e_2 $
  & 
  $E_3^t= -4(e_3+4e_4)$ & $E_4^t=16e_4$\\

  \hline
 $\cd 4{112}$ &$\to$ &  $\cd 4{81}$ & 
 $\lambda=\frac{1}{4} $ &$\af=0 $ &  $\bt=2$  &$\gamma =0$  \\
 &&& 
 $E_1^t= \sqrt{\frac{2}{t}}(2e_1-e_2) $ &
  $E_2^t= 2\sqrt{2t}e_2 $
  & 
  $E_3^t= -4(e_3+4e_4)$ & $E_4^t=16e_4$\\

 \hline
 $\cd 4{102}$ &$\to$ &  $\cd 4{82}$ & 
 $\lambda=\frac{1}{4} $ &$\af=\sqrt{2}t $  \\
 &&& $E_1^t= -\sqrt{2}(2e_1-e_2)$ &$E_2^t= \sqrt{2} te_2$ & $E_3^t= 2t (e_3-2e_4) $ & $E_4^t= 4t^2 e_4$\\
 
 \hline
 $\cd 4{112}$ &$\to$ &  $\cd 4{83}(\A)$ & 
 $\lambda=\frac{1}{4} $ &$\af=\A^{-\frac 12}$ &  $\bt=0$  &$\gamma =0$  \\
 &&& 
 $E_1^t= 2\A^{-\frac 12} e_2 $ &
  $E_2^t= t^{-1}(e_1 - \frac 12 e_2) $
  & 
  $E_3^t= \A^{-\frac 12} t^{-1} e_3$ & $E_4^t=t^{-2}\A^{-1}e_4$\\
  
 \hline
 $\cd 4{83}$ &$\to$ &  $\cd 4{84}$ & 
   &$\af= t^{-2} $\\
 &&& $E_1^t=  t^{-1} e_1+t e_2$ &$E_2^t= e_2$ & $E_3^t=  t^{-1}e_3 $ & $E_4^t= t^{-2}  e_4$\\
  
 \hline
 $\cd 4{112}$ &$\to$ &  $\cd 4{85}$ & 
 $\lambda=\frac{1}{4} $ &$\af=0 $ & $\bt=0  $ & $\gamma=0  $\\
 &&& $E_1^t=  t^{-2}  e_1-\frac{1}{2t}e_2$ &$E_2^t= t^{-1}e_2$ & $E_3^t= -\frac{1}{2t^3} e_3 $ & $E_4^t= \frac{1}{4t^6} e_4$\\
 
\hline
 $\cd 4{12}$ &$\to$ &  $\cd 4{86}$ & 
 $\lambda=-1 $ &$\af=\frac{3t^3}{2}$ \\
 &&& $E_1^t= t  e_1+e_2$ &$E_2^t= 2e_2+3t e_4$ & $E_3^t= 2t e_3-3t^3e_4 $ & $E_4^t=-6t^2  e_4$\\
 
 \hline
 $\cd 4{112}$ &$\to$ &  $\cd 4{87}(\Lambda)$ & 
 $\lambda=\Lambda - \Psi t$ &$\af=\Xi((\Theta-\Psi)t^{-1} - 1) $ & $\bt=0$ & $\gamma=\frac\Xi{\Psi t}(\Theta-\Psi)$\\
    \multicolumn{3}{l|}{$\Xi=\Theta^{-\frac 12}$}& $E_1^t= \sqrt\Theta t^{-1} (e_1 - \Psi e_2)$ &$E_2^t= \Xi t^{-1}(e_1 - (t + \Psi) e_2)$ & $E_3^t= -t^{-2}(te_3-\Xi^2e_4)$ & $E_4^t= t^{-2}e_4$\\
 
 \hline
 $\cd 4{112}$ &$\to$ &  $\cd 4{88}(\Lambda)$ & 
 $\lambda=\Lambda - \Theta t$ &$\af=-\Psi^{-\frac 12}$ & $\bt=0$ & $\gamma=0$\\
    &&& $E_1^t= \sqrt\Psi t^{-1} (e_1 - \Theta e_2)$ &
    \multicolumn{2}{l}{$E_2^t= \Psi^{-\frac 12} t^{-1}(e_1 - (t + \Theta) e_2)$}\\&& & $E_3^t= -t^{-2}(te_3-\Psi^{-1}e_4)$ & $E_4^t= t^{-2}e_4$\\
    
 \hline
 $\cd 4{112}$ &$\to$ &  $\cd 4{89}(\Lambda)$ & 
 $\lambda=\Lambda - \Psi t$ &$\af=-\Theta^{-\frac 12}$ & $\bt=0$ & $\gamma=0$\\
    &&& $E_1^t= \sqrt\Theta t^{-1} (e_1 - \Psi e_2)$ &
    \multicolumn{2}{l}{$E_2^t= \Theta^{-\frac 12} t^{-1}(e_1 - (t + \Psi) e_2)$}\\&& & $E_3^t= -t^{-2}(te_3-\Theta^{-1}e_4)$ & $E_4^t= t^{-2}e_4$\\
    
 \hline
 $\cd 4{112}$ &$\to$ &  $\cd 4{90}(\Lambda)$ & 
 $\lambda=\Lambda - \Theta t$ &$\af=\Xi((\Psi-\Theta)t^{-1} - 1) $ & $\bt=0$ & $\gamma=\frac\Xi{\Theta t}(\Psi-\Theta)$\\
    \multicolumn{3}{l|}{$\Xi=\Psi^{-\frac 12}$}& $E_1^t= \sqrt\Psi t^{-1} (e_1 - \Theta e_2)$ &$E_2^t= \Xi t^{-1}(e_1 - (t + \Theta) e_2)$ & $E_3^t= -t^{-2}(te_3-\Xi^2e_4)$ & $E_4^t= t^{-2}e_4$\\
    
 \hline
 $\cd 4{112}$ &$\to$ &  $\cd 4{91}(\Lambda,\A)$ & 
 $\lambda=\Lambda - \Psi t $ &$\af=-\A\Xi(1 + \Psi t^{-1})  $ & $\bt=0$ & $\gamma=-\A\Xi t^{-1}$\\
   \multicolumn{3}{l|}{$\Xi=\Theta^{-\frac 12}$}& $E_1^t= \sqrt\Theta t^{-1} (e_1 - \Psi e_2)$ &$E_2^t= \Xi t^{-1}(e_1 - (t + \Psi) e_2)$ & $E_3^t= -t^{-2}(te_3-\Xi^2e_4)$ & $E_4^t= t^{-2}e_4$\\
 
 \hline
 $\cd 4{112}$ &$\to$ &  $\cd 4{92}(\Lambda,\A)$ & 
 $\lambda=\Lambda - \Theta t $ &$\af=-\A\Xi(1 + \Theta t^{-1})  $ & $\bt=0$ & $\gamma=-\A\Xi t^{-1}$\\
  \multicolumn{3}{l|}{$\Xi=\Psi^{-\frac 12}$}& $E_1^t= \sqrt\Psi t^{-1} (e_1 - \Theta e_2)$ &$E_2^t= \Xi t^{-1}(e_1 - (t + \Theta) e_2)$ & $E_3^t= -t^{-2}(te_3-\Xi^2e_4)$ & $E_4^t= t^{-2}e_4$\\
 
 \hline
 $\cd 4{112}$ &$\to$ &  $\cd 4{93}(\A)$ & 
 $\lambda=0 $ &$\af=-\A $ & $\bt=t^{-2}  $ & $\gamma=0  $\\
 &&& $E_1^t= t^{-1}e_1$ &$E_2^t= t^{-1}e_1 - e_2$ & $E_3^t= -t^{-1}e_3-t^{-3}e_4$ & $E_4^t= t^{-2}e_4$\\
 
 \hline
 $\cd 4{112}$ &$\to$ &  $\cd 4{94}(\A,\B)$ & 
 $\lambda=0 $ &$\af=-\A $ & $\bt=-t^{-2}((\B-1)t - \B)$ & $\gamma=-\A t^{-1}  $\\
 &&& $E_1^t= t^{-1}e_1$ &$E_2^t= t^{-1}e_1 - e_2$ & $E_3^t= -t^{-1}e_3-\B t^{-3}e_4$ & $E_4^t= t^{-2}e_4$\\
 
 \hline
 $\cd 4{112}$ &$\to$ &  $\cd 4{95}(\A)$ & 
 $\lambda=t^2 $ &$\af=\A $ & $\bt=t^{-1}  $ & $\gamma=0  $\\
 &&& $E_1^t=  -t^{-1}e_1$ &$E_2^t= -t^{-1}e_1 + e_2$ & $E_3^t= -t^{-1}e_3 $ & $E_4^t= t^{-2}e_4$\\
 
 \hline
 $\cd 4{112}$ &$\to$ &  $\cd 4{96}(\A)$ & 
 $\lambda=0 $ &$\af=t $ & $\bt=-t^{-2}((\A - 1)t - \A)  $ & $\gamma=t^{-1}  $\\
 &&& $E_1^t=  t^{-1}e_1$ &$E_2^t= t^{-1}e_1 - e_2$ & $E_3^t= -t^{-1}e_3 - \A t^{-3} e_4$ & $E_4^t= t^{-2}e_4$\\
 
 \hline
 $\cd 4{112}$ &$\to$ &  $\cd 4{97}(\Lambda)$ & 
 $\lambda=\Lambda - \Psi t$ &$\af=-\Xi \sqrt{\Theta t}$ &  $\bt=(1-\Theta\Psi^{-1})\Xi^2$  &$\gamma =0$  \\
 \multicolumn{3}{l|}{$\Xi=(t + \Psi-\Theta)^{-\frac 12}$}& 
 \multicolumn{1}{l}{$E_1^t=  \Xi \sqrt{\Theta t^{-1}} (e_1 -\Psi e_2) $} &
  \multicolumn{2}{l}{$E_2^t= \Xi (\Theta t)^{-\frac 12} (e_1 - (t + \Psi) e_2)$}  \\&&& 
   \multicolumn{1}{l}{$E_3^t= \Xi^2(-e_3 + \Xi^2\Psi^{-1} e_4)$} & $E_4^t=\Xi^2 e_4$\\
 
 \hline
 $\cd 4{112}$ &$\to$ &  $\cd 4{98}(\Lambda)$ & 
 $\lambda=\Lambda - \Theta t$ &$\af=-\Xi \sqrt{\Psi t}$ &  $\bt=(1-\Psi\Theta^{-1})\Xi^2$  &$\gamma =0$  \\
 \multicolumn{3}{l|}{$\Xi=(t + \Theta-\Psi)^{-\frac 12}$}& 
 \multicolumn{1}{l}{$E_1^t=  \Xi \sqrt{\Psi t^{-1}} (e_1 -\Theta e_2) $} &
  \multicolumn{2}{l}{$E_2^t= \Xi (\Psi t)^{-\frac 12} (e_1 - (t + \Theta) e_2)$}  \\&&& 
   \multicolumn{1}{l}{$E_3^t= \Xi^2(-e_3 + \Xi^2\Theta^{-1} e_4)$} & $E_4^t=\Xi^2 e_4$\\
   
 \hline
 $\cd 4{112}$ &$\to$ &  $\cd 4{99}(\A)$ & 
 $\lambda=\frac t{(t + 1)^2} $ &$\af= \frac{t - \A + 1}{\sqrt{t - 1}}$ &  $\bt=1 + \frac 1t$  &$\gamma =(1-\A)\frac\Xi t$  \\
 \multicolumn{3}{l|}{$\Xi=\frac{t+1}{\sqrt{t-1}}$}& 
 $E_1^t=  \Xi(-(1 + \frac 1t) e_1+e_2)$ & $E_2^t= \Xi(t + 1) (-\frac 1t e_1+ e_2)$ & $E_3^t= \Xi^2(-e_3 + \frac 1t e_4)$ & $E_4^t=\Xi^4 e_4$\\
 
 \hline
 $\cd 4{112}$ &$\to$ &  $\cd 4{100}(\A)$ & 
 $\lambda=-\frac t2 + \frac 14 $ &$\af=-\sqrt 2\A $ & $\bt=0$ & $\gamma=0$\\
 &&& $E_1^t=  \sqrt 2 t^{-1}(\frac 12 e_1-\frac 14 e_2)$ &
 \multicolumn{2}{l}{$E_2^t= \sqrt 2 t^{-1} (e_1 -(t + \frac 12)e_2)$}\\&& & $E_3^t= -t^{-1}e_3 + 2 t^{-2} e_4$ & $E_4^t= t^{-2}e_4$\\
 
 \hline
 $\cd 4{112}$ &$\to$ &  $\cd 4{101}(\A,\B)$ & 
 $\lambda=\Xi^{-2} t$ &$\af= -\B\sqrt t$ &  $\bt=\Xi t^{-1}$  &$\gamma =(\A - \B) \frac\Xi {\sqrt t}$  \\
 \multicolumn{3}{l|}{$\Xi=t+1$}& 
 $E_1^t=  \Xi t^{-\frac 12}(-\Xi e_1+t e_2)$ & $E_2^t= -\Xi^2 t^{-\frac 12}e_1$ & $E_3^t= \Xi^2 (e_3 + \Xi^2 t^{-1} e_4)$ & $E_4^t=\Xi^4 e_4$\\
 
  \hline
 $\cd 4{112}$ &$\to$ &  $\cd 4{102}(\Lambda,\A)$ & 
 $\lambda=\Lambda + \Psi\Theta^{-1}t $ &$\af=\A\Lambda\Theta^{\frac 12}\Xi t^{-\frac 12}$ &  $\bt=(\Theta -2\Lambda)\Xi^2\Psi^{-1}$  &$\gamma =\A\Theta^{\frac 32}\Xi t^{-\frac 12}$  \\
 \multicolumn{3}{l|}{$\Xi=(\Theta -2\Lambda + t)^{-\frac 12}$}& 
 \multicolumn{1}{l}{$E_1^t=  \Theta^{\frac 12}\Xi t^{-\frac 12}(\Theta e_1 -\Lambda e_2) $} &
  \multicolumn{2}{l}{$E_2^t= \Theta^{-\frac 12}\Xi t^{-\frac 12} (\Theta e_1 + (t-\Lambda) e_2)$} \\
  &&& 
   \multicolumn{1}{l}{$E_3^t= \Theta\Xi^2 (e_3 +\Theta\Psi^{-1}\Xi^2 e_4)$} & $E_4^t=\Theta^2\Xi^2 e_4$\\
   
 \hline
 $\cd 4{102}$ &$\to$ &  $\cd 4{103}$ & 
 $\lambda=0 $ &$\af=0 $ &   &  \\
 &&& $E_1^t= - e_1+te_3$ &$E_2^t= -e_2$ & $E_3^t=  e_3 $ & $E_4^t=  e_4$\\

 \hline
 $\cd 4{112}$ &$\to$ &  $\cd 4{104}$ & 
 $\lambda= 0$ &$\af= t^{-1}$ & $\bt=t^{-2}  $ & $\gamma=t^{-1}  $\\
 &&& $E_1^t= t^{-1}  e_1$ &$E_2^t=t^{-1} e_2$ & $E_3^t= t^{-2} e_3 $ & $E_4^t= t^{-4}  e_4$\\

 \hline
 $\cd 4{112}$ &$\to$ &  $\cd 4{105}(\Lambda, \A, \B)$ & 
 $\lambda=\Lambda $ &$\af=t+t^{-1} $ & $\bt=\B+\A t^{-2}  $ & $\gamma=\B t^{-1}  $\\
 &&& $E_1^t= t^{-1}  e_1-e_3$ &$E_2^t= t^{-1}e_2$ & $E_3^t=t^{-2}  e_3 $ & $E_4^t=t^{-4}  e_4$\\

\hline
 $\cd 4{112}$ &$\to$ &  $\cd 4{106}(\A)$ & 
 $\lambda=0 $ &$\af=t^{-1} $ & $\bt= \A t^{-2}  $ & $\gamma=0  $\\
 &&& $E_1^t=t^{-1}   e_1$ &$E_2^t= t^{-1}e_2$ & $E_3^t=  t^{-2}e_3$ & $E_4^t=t^{-4}  e_4$\\

 \hline
 $\cd 4{112}$ &$\to$ &  $\cd 4{107}(\Lambda)$ & 
 $\lambda= \Lambda$ &$\af=\frac{1+\sqrt{1-4\Lambda}}{2t} $ & $\bt=0  $ & $\gamma= t^{-1} $\\
 &&& $E_1^t=t^{-1}   e_1$ &$E_2^t=t^{-1} e_2$ & $E_3^t=t^{-2}  e_3$ & $E_4^t=t^{-4}  e_4$\\

 \hline
 $\cd 4{112}$ &$\to$ &  $\cd 4{108}(\Lambda)$ & 
 $\lambda=\Lambda $ &$\af=\frac{1-\sqrt{1-4\Lambda}}{2t} $ & $\bt=0  $ & $\gamma= t^{-1} $\\
 &&& $E_1^t=t^{-1}   e_1$ &$E_2^t=t^{-1} e_2$ & $E_3^t=t^{-2}  e_3$ & $E_4^t=t^{-4}  e_4$\\
 \hline

 \hline
 $\cd 4{112}$ &$\to$ &  $\cd 4{109}(\Lambda, \A)$ & 
 $\lambda=\Lambda$ &$\af=0$ & $\bt=t^{-2}$ & $\gamma= \A t^{-1}$\\
 &&& $E_1^t=t^{-1} e_1$ &$E_2^t=t^{-1} e_2$ & $E_3^t=t^{-2} e_3$ & $E_4^t=t^{-4} e_4$\\
 \hline

 \hline
 $\cd 4{112}$ &$\to$ &  $\cd 4{110}(\Lambda)$ & 
 $\lambda=\Lambda$ &$\af=0$ & $\bt=0$ & $\gamma=t^{-1}$\\
 &&& $E_1^t=t^{-1} e_1$ &$E_2^t=t^{-1} e_2$ & $E_3^t=t^{-2} e_3$ & $E_4^t=t^{-4} e_4$\\
 \hline

 \hline
 $\cd 4{112}$ &$\to$ &  $\cd 4{111}(\Lambda)$ & 
 $\lambda=\Lambda$ &$\af=0$ & $\bt=\Lambda^{-1}$ & $\gamma=0$\\
 &&& $E_1^t=t^{-1} e_1$ &$E_2^t=t^{-1} e_2$ & $E_3^t=t^{-2} e_3+t^{-2}  e_4$ & $E_4^t=t^{-4} e_4$\\
 \hline 
 
\end{longtable}
}

\end{proof}

\end{document}